\documentclass[a4paper,11pt]{amsart}
\usepackage{a4wide}
\usepackage{eucal}
\usepackage{enumerate} 
\usepackage[usenames, dvipsnames]{color}
\usepackage{aliascnt}
\usepackage[utf8]{inputenc}
\usepackage[T1]{fontenc}
\usepackage[bookmarks=true,pdfstartview=FitH, pdfborder={0 0 0}, colorlinks=true,citecolor=red, linkcolor=blue]{hyperref}
\usepackage{caption}
\usepackage{subcaption}
\usepackage{tikz}
\usetikzlibrary{arrows}

\numberwithin{equation}{section}

\advance\textheight by 2cm

\theoremstyle{plain}
\newtheorem{thm}{Theorem}[section]
\newaliascnt{cor}{thm}
\newaliascnt{prop}{thm}
\newaliascnt{lem}{thm}
\newtheorem{cor}[cor]{Corollary}

\newtheorem{lem}[lem]{Lemma}
\aliascntresetthe{cor}
\aliascntresetthe{prop}
\aliascntresetthe{lem} 
\theoremstyle{definition}
\newaliascnt{defn}{thm}
\newaliascnt{asu}{thm}

\newtheorem{asu}[asu]{Assumption}
\aliascntresetthe{defn}
\aliascntresetthe{asu}
\theoremstyle{remark}
\newaliascnt{rem}{thm}
\newaliascnt{exa}{thm}
\newtheorem{rem}[rem]{Remark}
\newtheorem{exa}[exa]{Example}
\aliascntresetthe{rem}
\aliascntresetthe{exa}
\newenvironment{psmallmatrix}
{\left(\begin{smallmatrix}}
{\end{smallmatrix}\right)}

\newcommand{\RR}{\mathbb{R}}
\newcommand{\CC}{\mathbb{C}}
\newcommand{\BB}{\mathbb{B}}
\newcommand{\NN}{\mathbb{N}}
\newcommand{\eps}{\varepsilon}
\newcommand{\sL}{\mathcal{L}}
\newcommand{\sB}{\mathcal{B}}
\newcommand{\sQ}{\mathcal{Q}}
\newcommand{\sR}{\mathcal{R}}
\newcommand{\sRtn}{\mathcal{R}_{\tn}}

\newcommand{\sx}{{\scriptstyle\mathcal{X}}}
\newcommand{\mv}{\mathsf{v}}
\newcommand{\me}{\mathsf{e}}
\newcommand{\mei}{\me^i}
\newcommand{\mee}{\me^e}

\newcommand{\rL}{\mathrm{L}}
\newcommand{\rM}{\mathrm{M}}
\newcommand{\rC}{\mathrm{C}}
\newcommand{\rW}{\mathrm{W}}
\newcommand{\p}{{\raisebox{1.3pt}{{$\scriptscriptstyle\bullet$}}}}
\newcommand{\ps}{{\raisebox{0.3pt}{{$\scriptscriptstyle\bullet$}}}}

\newcommand{\RRp}{{\RR_+}}
\newcommand{\LpneCm}{\mathchoice{\rL^p\bigl([0,1],\CC^m\bigr)}{\rL^p([0,1],\CC^m)}{\rL^p([0,1],\CC^m)}{\rL^p([0,1],\CC^m)}}
\newcommand{\CneCm}{\mathchoice{{\rC\bigl([0,1],\CC^m\bigr)}}{{\rC([0,1],\CC^m)}}{{\rC([0,1],\CC^m)}}{{\rC([0,1],\CC^m)}}}
\newcommand{\WepneCm}{\mathchoice{\rW^{1,p}\bigl([0,1],\CC^m\bigr)}{\rW^{1,p}([0,1],\CC^m)}{\rW^{1,p}([0,1],\CC^m)}{\rW^{1,p}([0,1],\CC^m)}}
\newcommand{\LpRpCl}{\mathchoice{\rL^p\bigl(\RRp,\CC^{\ell}\bigr)}{\rL^p(\RRp,\CC^{\ell})}{\rL^p(\RRp,\CC^{\ell})}{\rL^p(\RRp,\CC^{\ell})}}
\newcommand{\CnRpCl}{\mathchoice{\rC_0\bigl(\RRp,\CC^{\ell}\bigr)}{\rC_0(\RRp,\CC^{\ell})}{\rC_0(\RRp,\CC^{\ell})}{\rC_0(\RRp,\CC^{\ell})}}
\newcommand{\WepRpCl}{\mathchoice{\rW^{1,p}\bigl(\RRp,\CC^{\ell}\bigr)}{\rW^{1,p}(\RRp,\CC^{\ell})}{\rW^{1,p}(\RRp,\CC^{\ell})}{\rW^{1,p}(\RRp,\CC^{\ell})}}

\newcommand{\Tt}{(T(t))_{t\ge0}}

\newcommand{\dt}{\,dt}
\newcommand{\dr}{\,dr}

\newcommand{\ds}{\,ds}
\newcommand{\dds}{\frac{d}{ds}}
\newcommand{\tdds}{\tfrac{d}{ds}}

\newcommand{\rank}{\operatorname{rank}}
\newcommand{\diag}{\operatorname{diag}}
\newcommand{\rg}{\operatorname{rg}}

\renewcommand{\r}{\right}
\renewcommand{\l}{\left}

\newcommand{\ft}{{\tilde f}}
\newcommand{\Cb}{{\bar C}}
\newcommand{\Gt}{{\tilde G}}
\newcommand{\Gb}{{\bar G}}
\newcommand{\Id}{Id}
\newcommand{\IId}{\diag(\Id,\Id)}
\newcommand{\sBt}{\sB_t}
\newcommand{\Bttn}{(\sB_t)_{t\in[0,\tn]}}
\newcommand{\tn}{t_0}
\newcommand{\tmax}{t_{\mathrm{max}}}
\newcommand{\dX}{{\partial X}}
\newcommand{\phib}{{\bar\varphi}}

\newcommand{\cb}{{\bar c}}
\newcommand{\vh}{{\hat v}}
\newcommand{\Phib}{{\bar\Phi}}
\newcommand{\Rt}{R_t}
\newcommand{\St}{S_t}
\renewcommand{\theta}{\vartheta}
\renewcommand{\u}{q}

\newcommand{\Xe}{X^e}
\newcommand{\Xin}{X^i}

\newcommand{\Qt}{Q_t}
\newcommand{\epsl}{\eps_\lambda}
\newcommand{\Ll}{L_\lambda}
\newcommand{\Phie}{\Phi^e}
\newcommand{\Phii}{\Phi^i}
\newcommand{\Phin}{\Phi_0}
\newcommand{\Phien}{\Phi^e_0}
\newcommand{\Phiin}{\Phi^i_0}
\newcommand{\qe}{q^e}
\newcommand{\qi}{q^i}
\newcommand{\ce}{\lambda^e}
\newcommand{\ci}{\lambda^i}
\newcommand{\cij}{\ci_j}
\newcommand{\cek}{\ce_k}
\newcommand{\phiek}{\phi_k}
\newcommand{\phiij}{\varphi_j}

\newcommand{\fe}{f^e}
\newcommand{\fin}{f^i}
\newcommand{\ue}{u^e}
\newcommand{\ui}{u^i}

\newcommand{\ueh}{\uh^e}
\newcommand{\uih}{\uh^i}
\newcommand{\Phibe}{\Phib^e}
\newcommand{\Phibi}{\Phib^i}
\newcommand{\Dme}{D_m^e}

\newcommand{\Dne}{D_0^e}

\newcommand{\Vne}{V_0^e}
\newcommand{\Vni}{V_0^i}
\newcommand{\Vei}{V_1^i}
\newcommand{\Te}{T^e}
\newcommand{\Ti}{T^i}
\newcommand{\uh}{{\hat u}}
\newcommand{\fh}{{\hat f}}
\newcommand{\Pep}{P^e_+}
\newcommand{\Pem}{P^e_-}
\newcommand{\Pip}{P^i_+}
\newcommand{\Pim}{P^i_-}
\newcommand{\dXe}{\partial\Xe}
\newcommand{\dXi}{\partial\Xin}
\newcommand{\Ae}{A^e}
\newcommand{\Ai}{A^i}
\newcommand{\Aem}{A_m^e}
\newcommand{\Aim}{A_m^i}
\newcommand{\Le}{L^e}
\newcommand{\Li}{L^i}
\newcommand{\Tet}{(T^e(t))_{t\ge0}}
\newcommand{\Tle}{T^e_{l}}
\newcommand{\Tlet}{(\Tle(t))_{t\ge0}}
\newcommand{\Tre}{T^e_r}
\newcommand{\Tret}{(\Tre(t))_{t\ge0}}
\newcommand{\Titt}{(\Ti(t))_{t\in\RR}}
\newcommand{\sBte}{\sB^e_t}
\newcommand{\sBti}{\sB^i_t}
\newcommand{\Jp}{{J_\phi}}
\newcommand{\Jpb}{{J_\phib}}
\newcommand{\Jpi}{{J_{\phi}^{-1}}}
\newcommand{\Jpbi}{{J_{\phib}^{-1}}}
\newcommand{\cbm}{{|\cb|}}
\newcommand{\n}{q}
\newcommand{\B}{B}
\title [Flows on Metric Graphs]
{Flows on Metric Graphs with General Boundary Conditions }
\author{Klaus-Jochen Engel}
\address{University of L'Aquila, Department of Information Engineering, Computer Science and Mathematics, Via Vetoio, Coppito, I-67100 L'Aquila (AQ), Italy. \tt{klaus.engel@univaq.it}
}
\author{Marjeta Kramar Fijav\v{z}}
\address{University of Ljubljana, Faculty of Civil and Geodetic Engineering, Jamova 2, SI-1000 Ljubljana, Slovenia /
Institute of Mathematics, Physics, and Mechanics,
Jadranska 19, SI-1000 Ljubljana, Slovenia.
\tt{marjeta.kramar@fgg.uni-lj.si}
}

\begin{document}

\begin{abstract}
In this note we study the generation of $C_0$-semigroups by first order differential operators on $\LpRpCl\times\LpneCm$ with general boundary conditions. In many cases we are able to characterize the generation property in terms of the invertibility of a matrix associated to the boundary conditions. The abstract results are used to study well-posedness of transport equations on non-compact metric graphs.
\end{abstract}
\maketitle

\section{Introduction}

This is a companion paper of \cite{EKF:19} where we studied linear wave and diffusion equations on metric graphs subject to general boundary conditions. 
The aim of the present work is to apply the same approach used there to examine first order equations modeling flows on metric graphs. 
This leads to the abstract Cauchy problem
\begin{equation}\label{eq:acp}
\begin{cases}
\dot x(t)= G x(t),&t\ge0,\\
x(0)=x_0,
\end{cases}
\end{equation}
for a first-order differential operator $G\colon D(G)\subset X\to X$ on the Banach space 
\[X=\LpRpCl\times\LpneCm.\] 
Here, $\RR_+$ and the interval $[0,1]$ represent the unbounded, respectively normalized bounded edges of the graph while its structure is encoded in the boundary conditions appearing in the definition of the domain $D(G)$.

In the simplest case, all edges are bounded, hence we have $\ell=0$, i.e., $X=\LpneCm$ and 
\begin{equation}\label{eq:example-intro}
\begin{aligned}
G&=c(\p)\cdot\tdds,\\
D(G)&=\bigl\{f\in\WepneCm:V_0f(0)=V_1f(1)
\bigr\},
\end{aligned}
\end{equation}
where $c(s)=\diag(\lambda_1(s),\ldots,\lambda_m(s))$
for non-vanishing, Lipschitz continuous functions $\lambda_k(\p)$, $k=1,\ldots,m$, and suitable scalar ``boundary'' matrices $V_0,V_1\in\rM_m(\CC)$. 

It is well known, for the details we refer to  \cite[Sect.~II.6]{EN:00}, that the problem \eqref{eq:acp} is well-posed if and only if $G$ generates a strongly continuous semigroup on $X$. 

\smallskip
Our results now give sufficient conditions, which in many important cases are also necessary, in terms of the invertibility of a linear map, for a first-order differential operator $G$ with general boundary conditions to be a generator, hence for the well-posedness of the associated problem \eqref{eq:acp}. For example, by  \autoref{cor:gen-matrix-det} 
 the operator $G$ in \eqref{eq:example-intro} is a generator if and only if
\begin{equation}\label{eq:det=/0}
\det\bigl(V_1P_++V_0P_-)\ne0,
\end{equation}
where $P_\pm$ denote the spectral projections associated to the positive/negative eigenvalues of $c(\p)$. In particular, condition~\eqref{eq:det=/0} characterizes the well-posedness of the Cauchy problem \eqref{eq:acp} for $G$ given by \eqref{eq:example-intro}.

The semigroup approach to linear transport equations (or flows) on networks was initiated in \cite{KS:05} and further pursued by many authors, see e.g.~\cite{MS07,EKNS08, BDK:13,Rad08, BN:14, BKFR:17,BFN:16,BP19}.
In all these works  finite compact graphs (i.e., with finitely many edges that are all compact intervals) are considered, $c(\p)$ are constant and diagonal matrices, and the boundary conditions rely on the structure of the graph, representing conservation of the mass, so-called Kirchhoff conditions. In \cite{Dor08, DKS09, DFKNR:10} the same approach was applied to study flows on infinite but compact graphs.

The setting has been generalized in \cite{Klo:10,Eng:13} to so-called difference operators, that are operators of the form  \eqref{eq:example-intro} where $c(\p)$ was considered to be a constant, not necessarily diagonal, self-adjoint matrix.  
We note that our problem can also be formulated as a first order boundary system or a so-called port-Hamiltonian system. In this framework 
generation results are shown in Hilbert space setting by a suitable basis transformation and then carried over to $\rL^p$-spaces,  see for example \cite{ZwaLeGMasVil10,JMZ:15}.
These methods were also applied in  \cite{JK:19,JW:19} to first order equations on some classes of non-compact graphs.

\smallskip
In comparison to the above cited works our approach presents several novelties: it allows to 
\begin{itemize}
\item consider operators on $\LpRpCl\times\LpneCm$ with applications to \emph{non}-compact graphs,
\item treat very general (also non-diagonal) ``velocity matrices'' $c(\p,\p)$, cf.~\eqref{eq:rep-a(s)},
\item treat very general (also non-local) boundary conditions, cf.~\eqref{eq:def-G-general},
\item treat all cases of $p\in[1,+\infty)$ simultaneously without using interpolation arguments,
\item obtain necessary conditions for the generator property, hence the well-posedness, for important classes of boundary conditions.
\end{itemize}
Our reasoning is based on similarity transformations and a  perturbation result for boundary perturbations of domains of generators from \cite[Sect.~4.3]{ABE:13}. 

\smallskip
This paper is organized as follows. In \autoref{sec:Lp-gen}, after introducing our general setup in \autoref{sec:intro}, we state and prove in \autoref{sec:m-gen-res} the main generation result, \autoref{thm:gen-Lp}. 
In \autoref{sec:necessary} we supplement this result in \autoref{thm:NC-1} by a necessary condition for $G$ to be a generator.
If the boundary operators defining the domain $D(G)$ are, modulo bounded perturbations, given by point evaluations, this gives simple generation conditions which are presented in \autoref{subcsec:reformulation}.  In the compact case $\ell =0$ this reduces in \autoref{cor:gen-matrix-det} to a very simple necessary and sufficient determinant condition like \eqref{eq:det=/0}. 
In \autoref{sec:networks} we then show how our results apply to transport processes on metric graphs. In particular, we reproduce and generalize many existing results which demonstrates the versatility of our approach. 
Finally, in the appendix we recall our approach to boundary perturbations  from  \cite{ABE:13,EKF:19} and present in \autoref{thm:NC-gen} a new necessary condition for a perturbed generator to be again a generator.

\section{Generation of $C_0$-Semigroups}\label{sec:Lp-gen}

In this section we introduce our general setup, present the generation results, give a necessary condition for generation and then give some reformulations of the boundary conditions.

\subsection{The Setup}\label{sec:intro}

Let $\RRp:=[0,+\infty)$. Throughout this section we make the following 

\begin{asu}\label{asu:s-asu-Lp} 
Consider for some fixed $p\ge1$ and $\ell,m\in\NN_0$ satisfying $\ell + m>0$
\begin{itemize}
\item the space $\Xe:=\LpRpCl$ of functions on the $\ell$ \emph{``external edges''},
\item the space $\Xin:=\LpneCm$ of functions on the $m$ \emph{``internal edges''},
\item the \emph{``state space''} $X:=\Xe\times\Xin$ of functions on all $\ell + m$ edges,
\item a \emph{``boundary space''} $\dX\subseteq\CC^{\ell + m }$ (which will be determined later, cf.~\eqref{eq:def-dX}),
\item  a \emph{``boundary operator''} $\Phi=(\Phie,\Phii)\in\sL(\CnRpCl\times\CneCm,\dX)$ (which determines the boundary conditions),
\item  a  \emph{``velocity matrix''} $c(\p,\p)\in\rL^\infty(\RRp,\rM_{\ell}(\CC))\times\rL^\infty([0,1],\rM_m(\CC))$ (containing the coefficients of the associated differential equation) of the form
\begin{equation}\label{eq:rep-a(s)}
\begin{aligned}
c(\p,\p)&:=
\begin{psmallmatrix}
\qe(\ps)&0\\0&\qi(\ps)
\end{psmallmatrix}
\cdot
\begin{psmallmatrix}
\ce(\ps)&0\\0&\ci(\ps)
\end{psmallmatrix}
\cdot
\begin{psmallmatrix}
\qe(\ps)&0\\0&\qi(\ps)
\end{psmallmatrix}^{-1}\\
&=: q(\p,\p)\cdot\lambda(\p,\p)\cdot q(\p,\p)^{-1},
\end{aligned}
\end{equation}
where the diagonal entries of $\lambda(\p,\p)$ are matrix-valued functions given by
\begin{alignat*}{3}
&\ce(r)=\diag\bigl(\cek(r)\bigr)_{k=1}^{\ell} \in\rM_{\ell}(\RR),&\quad&r\in\RRp,\\
&\ci(s)=\diag\bigl(\cij(s)\bigr)_{j=1}^m\in\rM_m(\RR),&\quad&s\in[0,1]
\end{alignat*}
and have strictly constant sign, i.e., there exists $\eps>0$ such that
\begin{alignat}{3}\label{eq:aek-beschr}
&\eps\le\cek(0)\cdot\cek(r)&\quad&\text{for all $k=1,\ldots,\ell$, $r\in\RRp$,}\\
&\eps\le\cij(0)\cdot\cij(s)&\quad&\text{for all $j=1,\ldots,m$, $s\in[0,1]$,}\notag
\end{alignat}
and the matrix-valued function $\u(\p,\p)$ is Lipschitz continuous and satisfies 
\begin{equation*}\label{eq:qq-1}
\u(\p,\p),\u(\p,\p)^{-1}\in\rL^\infty\bigl(\RRp,\rM_{\ell}(\CC)\bigr)\times\rL^\infty\bigl([0,1],\rM_m(\CC)\bigr).
\end{equation*}

\end{itemize}
\end{asu}

Note that these hypotheses imply that the mapping  $\lambda(\p,\p)$ is invertible and $\lambda(\p,\p),\lambda(\p,\p)^{-1}$ are both uniformly bounded.

\begin{rem}
We mention that in general even a smooth positive definite valued {$c(\p, \p)$} 
cannot be represented as in \eqref{eq:rep-a(s)}, cf. \cite[Rem.~2.9.(vi)]{EKF:19}.
However, if {$c(\p,\p)$} is analytic or if {$c(r,s)$} has $m+\ell$ distinct eigenvalues for all {$r\in\RRp, s\in[0,1]$} then it can always be written as in \eqref{eq:rep-a(s)}, see \cite[Chapt.~2]{Kat:80}. For the relevant assumptions in the Hilbert space setting we refer to \cite[Lem.~A.1]{KMN:21}.  
\end{rem}

In the sequel we equip all subspaces of $\CC^k$ with  the maximum norm $\|\p\|_\infty$, i.e., we define 
\[
\bigl\|(x_1,\ldots,x_k)^\top\bigr\|_{\CC^k}:=\max\bigl\{|x_1|,\ldots,|x_k|\bigr\}.
\]
Then, using that $\WepRpCl\subset\CnRpCl$ and $\WepneCm\subset\CneCm$  (see \cite[Sect.8.2]{Bre:11}), we define on $X=\LpRpCl\times\LpneCm$ the operator%
\begin{equation}\label{eq:def-G-general}
\begin{aligned}
G&:=c(\p,\p)\cdot\tdds,\\
D(G)&:=\bigl\{f\in\WepRpCl\times\WepneCm:\Phi f=0\bigr\}=\ker(\Phi).
\end{aligned}
\end{equation}
Our aim is to give conditions on the boundary operator $\Phi$ implying that $G$ generates a $C_0$-semigroup on $X$. As we will see in our main result, \autoref{thm:gen-Lp}, this can be achieved through the  invertibility of the operator $\sRtn$ defined in \eqref{eq:def-sR}.

In order to choose in \autoref{asu:s-asu-Lp} a reasonable boundary space $\dX$ note that for $\pm\dds$ to be a generator on $\rL^p(\RRp)$ we need to impose no boundary condition in case of ``$+$'' while for the sign ``$-$'' a boundary condition in $r=0$ is needed. On the other hand, for $\pm\dds$ to be a generator on $\rL^p[0,1]$ in both cases exactly one boundary condition is needed, in case of ``$+$'' in $s=1$, and for the sign ``$-$'' in $s=0$. 
Denote the (constant and diagonal) spectral projections corresponding to the positive/negative eigenvalues of $\ce(r)\in\rM_{\ell}(\RR)$ and $\ci(s)\in\rM_m(\RR)$, $r\in\RRp$, $s\in[0,1]$, by
\begin{equation}\label{eq:def-P+-Phib}
\Pep,\Pem\in\rM_{\ell}(\CC),\quad 
\Pip,\Pim\in\rM_m(\CC),
\end{equation}
respectively.
Then, by the previous considerations, for $G=c(\p,\p)\cdot\dds$ in \eqref{eq:def-G-general} to be a generator it is reasonable to impose $\rank(\Pem)+m$ boundary conditions. This motivates the choice of
\begin{equation}\label{eq:def-dX}
\dX:=\dXe\times\dXi:=\rg(\Pem)\times\CC^m\subseteq\CC^{\ell + m }
\end{equation}
as boundary space.

\smallbreak
In order to state our main result rigorously we need some more notation. 
For $1\le k\le \ell$ we define
\begin{equation}\label{eq:def-phik}
\phiek(r):=\int_0^r\frac{ds}{|\ce_k(s)|},\ r\in\RRp.
\end{equation}
It follows from \eqref{eq:aek-beschr} that all functions $\phiek:\RRp\to\RRp$ are Lipschitz continuous, surjective, and strictly increasing, hence invertible with Lipschitz continuous inverses $\phi_k^{-1}:\RRp\to\RRp$.
Next, for $1\le j\le m$ we similarly define
\begin{alignat}{5}\label{eq:def-varphi}
&\phiij(s):=\int_0^s\frac{dr}{\cij(r)},
\quad\cb_j:=\frac1{\phiij(1)},
&\quad&\phib_j(s):=\cb_j\cdot\phiij(s)
,\ s\in[0,1].
\end{alignat}
Also, all functions $\phib_j:[0,1]\to[0,1]$ are Lipschitz continuous, surjective, and strictly monotone, hence invertible with Lipschitz continuous inverses $\phib_j^{-1}:[0,1]\to[0,1]$.
Now, we put 
\begin{alignat*}{3}
&\phi:=(\phi_1,\ldots,\phi_{\ell})^\top:\RRp\to\RR_+^{\ell} ,&\qquad &\phi^{-1}:=(\phi_1^{-1},\ldots,\phi_{\ell}^{-1})^\top:\RRp\to\RR_+^{\ell} ,\\
&\phib:=(\phib_1,\ldots,\phib_m)^\top:[0,1]\to[0,1]^m, &&\phib^{-1}:=(\phib_1^{-1},\ldots,\phib_m^{-1})^\top:[0,1]\to[0,1]^m,\\
&\cb:=\bigl(\cb_1,\ldots,\cb_m\bigr)\in\RR^m, &&\cbm:=\bigl(|\cb_1|,\ldots,|\cb_m|\bigr)\in\RR_+^m.
\end{alignat*}
Finally, for functions $h=(h_1,\ldots,h_n)^\top:I\subseteq\RR\to\CC^n$, $\nu=(\nu_1,\ldots,\nu_n):I\to I^n$, a vector $x=(x_1,\ldots,x_n)^\top\in\RR^n$ and scalars $r,s\in\RR$ we set
\begin{align*}
(h\circ\nu)(r)&:=\bigl(h_1(\nu_1(r)),\ldots,h_n(\nu_n(r))\bigr)^\top,\\
h(r+x)&:=\bigl(h_1(r+x_1),\ldots,h_n(r+x_n)\bigr)^\top,\\
h\bigl(r+\tfrac{s}{x}\bigr)&:=\bigl(h_1\bigl(r+\tfrac{s}{x_1}\bigr),\ldots,h_n\bigl(r+\tfrac{s}{x_n}\bigr)\bigr)^\top.
\end{align*}

Using this notation we define the transformations
\begin{equation}\label{eq:def-Q_phi}
\begin{aligned}
&\Jp\in\sL(\Xe),&\quad&\Jp\fe:=\fe\circ\phi&\quad&\text{for }\fe\in\Xe,\\
&\Jpb\in\sL(\Xin), &&\Jpb\fin:=\fin\circ\phib&&\text{for }\fin\in\Xin.
\end{aligned}
\end{equation}
Then $\Jp$ and $\Jpb$ are invertible with inverses $\Jpi=J_{\phi^{-1}}\in\sL(\Xe)$ and $\Jpbi=J_{\phib^{-1}}\in\sL(\Xin)$.


Using these facts we define for fixed $\tn>0$ and $t\in[0,\tn]$ the bounded linear operators $\Qt\in\sL(\rL^p([0,\tn],\dXe),\Xe)$ and $\Rt,\St\in\sL(\rL^p([0,\tn],\dXi),\Xin)$ by
\begin{equation}\label{eq:def-Rt-Sr}
\Qt\ue:=\ueh(t-\p\bigr),\quad
\Rt\ui:=\uih\bigl(t-\tfrac\ps\cbm\bigr),
\quad
\St\ui:=\uih\bigl(t-\tfrac{1-\ps}\cbm\bigr),
\end{equation}
where $\uh$ denotes the extension of a function $u$ defined on $I\subset\RR$ to $\RR$ by the value $0$. 
In order to state the main result we need to define an operator $\sRtn$ as follows where
\begin{align*}
\Phib&:=(\Phibe,\Phibi):=\bigl(\Phie\cdot\qe(\p)\cdot\Jp,\Phii\cdot\qi(\p)\cdot\Jpb\bigr)\\
&\in\sL\bigl(\CnRpCl\times\CneCm,\dX\bigr).
\end{align*}

\begin{lem}\label{lem:bdd-ext}
Let $0<\tn\le\frac1{\|\cb\|_\infty}$. Then the operator $\sRtn:\rW_0^{1,p}([0,\tn],\dX)\subset\rL^p([0,\tn],\dX)\to\rL^p([0,\tn],\dX)$ given by
\begin{equation}\label{eq:def-sR}
\begin{aligned}
\Bigl(\sRtn\tbinom\ue\ui\Bigr)(t):&=
\bigl(\Phibe,\Phibi\bigr)\cdot
\begin{pmatrix}
\Qt&0\\
0&\Pip\St -\Pim\Rt
\end{pmatrix}
\tbinom\ue\ui
\\
&=\Bigl(\Phibe\cdot\Qt\ue,\Phibi\bigl(\Pip\cdot\St -\Pim\cdot\Rt\bigr)\ui\Bigr)
\end{aligned}
\end{equation}
for $t\in[0,\tn]$ and $(\ue,\ui)^\top\in\rW^{1,p}_0([0,\tn],\dXe)\times\rW^{1,p}_0([0,\tn],\dXi)$
is well-defined and has a unique bounded extension denoted again by 
$\sRtn\in\sL(\rL^p([0,\tn],\dX))$.
\end{lem}

\begin{proof} 
Define $\psi(s):=1-s$ for $s\in[0,1]$. Then $\St=\psi\cdot\Rt$, hence 
\begin{equation}\label{eq:Bti-psi}
\Pip\cdot\St -\Pim\cdot\Rt=(\Pip\cdot\psi -\Pim)\cdot\Rt
\end{equation}
Observe, that $\Phibe\in\sL\bigl(\CnRpCl,\dX\bigr)$ and $\Phibi(\Pip\cdot\psi -\Pim)\in\sL\l(\CneCm,\dX\r)$
are well-defined. Hence, it suffices to show that the operator
\begin{equation}\label{eq:def-Utn}
\begin{aligned}
&U_{\tn}^\Psi:\rW^{1,p}_0\bigl([0,\tn],\CC^k\bigr)\subset\rL^p\bigl([0,\tn],\CC^k\bigr)\to\rL^p\bigl([0,\tn],\dX\bigr),\\
&(U_{\tn}^\Psi u)(t):=\Psi\,\uh\bigl(t-\theta(\p)\bigr),\ t\in[0,\tn]
\end{aligned}
\end{equation}
is well-defined and has a bounded extension in $\sL(\rL^{p}([0,\tn],\CC^k),\rL^{p}([0,\tn],\dX))$ where for the
\begin{itemize}
\item external part $k=\ell$, $\Psi\in\sL(\rC_0(I,\CC^k),\dX)$ and $\theta(s):=s$, $s\in I:=\RRp$, 
\item internal part $k=m$, $\Psi\in\sL(\rC(I,\CC^k),\dX)$ and $\theta(s):=\tfrac s{\,\cbm\,}$, $s\in I:=[0,1]$.
\end{itemize}

Let
\begin{equation}\label{eq:def-I_tn}
I_{\tn}:=
\begin{cases}
[0,\tn],&\text{in the external case},\\
\bigl[0,\|\cb\|_\infty\cdot\tn\bigr],&\text{in the internal case.}
\end{cases}
\end{equation}
Then, by repeating the same arguments as in the proof of \cite[Lem.~2.2]{EKF:19}, it follows that $U_{\tn}^\Psi$ is indeed well-defined,  bounded,  and satisfies for $u\in\rW_0^{1,p}([0,\tn],\CC^k)$  the estimate
\begin{equation}
\|U_{\tn}^\Psi u\|_p^p \le\bigl(|\eta|(I_{\tn})\bigr)^{p}\cdot\|u\|^p_{p}.\label{eq:est-Utn}
\end{equation}
Here, $|\eta|:I\to\RRp$ denotes the positive Borel measure defined by the total variation of 
$\eta$ and $\eta:I\to\sL(\CC^k,\dX)$ is the function of bounded variation such that $\Psi$ is given by the Riemann--Stieltjes integral
\begin{equation}\label{eq:Phi-generic-e}
\Psi h=\int_I d\eta(s)\, h(s)
\quad\text{for all}\quad
h\in\rW^{1,p}\bigl(I,\CC^k\bigr).
\qedhere
\end{equation}
\end{proof}

\subsection{The Main Generation Results}\label{sec:m-gen-res}

We are now ready to state our main result on the generation of a $C_0$-semigroup by operators given by \eqref{eq:def-G-general}. 

\begin{thm}\label{thm:gen-Lp}
Let \autoref{asu:s-asu-Lp}  be satisfied.
If there exists $\tn>0$ such that the operator $\sRtn\in\sL(\rL^p([0,\tn],\dX))$ given by \eqref{eq:def-sR} is invertible,
then the operator $G$ defined in \eqref{eq:def-G-general} generates a $C_0$-semigroup on $X=\LpRpCl\times\LpneCm$.
\end{thm}

\begin{rem}\label{rem-sRt-tilde}
It is slightly inconvenient that $\sRtn$ is defined on functions assuming values in the subspace $\dX=\rg(\Pem)\times\CC^m$ of the (simpler) space $\CC^{\ell} \times\CC^m$. However, one can easily extend $\sRtn$ to an operator $\tilde\sR_{\tn}\in\sL(\rL^p([0,\tn],\CC^{\ell} \times\CC^m))$ without changing its invertibility. In fact, if we put
\begin{equation*} 
\tilde\sR_{\tn}:=
\begin{psmallmatrix}
\Pep&0\\0&0
\end{psmallmatrix}
+\sRtn\cdot
\begin{psmallmatrix}
\Pem&0\\0&\Id_{m}
\end{psmallmatrix}
\in\sL\bigl(\rL^p\bigl([0,\tn],\CC^{\ell} \times\CC^m\bigr)\bigr),
\end{equation*}
then $\sRtn$ is invertible if and only if $\tilde\sR_{\tn}$ is invertible.
\end{rem}

The proof of \autoref{thm:gen-Lp} is split into several steps. First, under the hypothesis $c(\p,\p)$ to be diagonal, that is $q(\p,\p)\equiv\IId$, we normalize the operator $G$ to obtain a simpler but similar operator $\Gb$ with constant velocities. The case of general $q(\p,\p)$ as in \eqref{eq:rep-a(s)} then follows by similarity and bounded perturbation.

\smallbreak
Recall that 
$\Jp\in\sL(\Xe)$, $\Jpb\in\sL(\Xin)$ are given by \eqref{eq:def-Q_phi}. Moreover, let 
\begin{equation}\label{eq:def-Cb}
\Cb:=\diag(\cb_1,\ldots,\cb_m)\in\rM_m(\RR)
\end{equation} 
where the $\cb_i$ are defined in \eqref{eq:def-varphi}.

\begin{lem}\label{lem:Gb}
Let $q(\p,\p)=\IId$. Then the operator $G$ given in \eqref{eq:def-G-general} on $X=\Xe\times\Xin$ is similar to $\Gb:D(\Gb)\subset X\to X$ where
\begin{equation}\label{eq:def-sGr}
\Gb:=
\diag\Bigl(\bigl(\Pep-\Pem\bigr)\cdot\tdds,\Cb\cdot\tdds\Bigr),
\quad
D(\Gb):=\ker(\Phib)
\end{equation}
and
$
\Phib:=(\Phie\cdot \Jp,\Phii\cdot \Jpb)
\in\sL(\CnRpCl\times\CneCm,\dX)
$.

\end{lem}

\begin{proof} 
Consider the invertible transformation
\begin{equation*}
S:=\diag\bigl(\Jp,\Jpb\bigr)
\in\sL(X)
\qquad\text{with inverse}\qquad
S^{-1}:=\diag\bigl(\Jpi,\Jpbi\bigr)
\in\sL(X).
\end{equation*}
We claim that
\begin{equation}\label{eq:sim-sG-sGr}
\Gb=S^{-1}\cdot G\cdot S.
\end{equation}
Let $\Gt:=S^{-1}GS$. Then, by definition, 
\[f\in D(\Gt) \iff S f\in D(G)=\ker(\Phi) \iff f\in\ker(\Phib),\]
i.e., $D(\Gt)=D(\Gb)$.
Next, $(\phi^{-1})'=|\ce(\p)|\circ\phi^{-1}$
implies that

\begin{equation*}
\Jp\cdot(\Pep-\Pem)\cdot\tdds\cdot \Jp^{-1}=
\ce(\p)\cdot\tdds.
\end{equation*}
Similarly, since $(\Cb\cdot\phib^{-1})'=\ci(\p)\circ\phib^{-1}$ we obtain
\begin{equation*}
\Jpb\cdot\Cb\cdot\tdds\cdot \Jpb^{-1}=
\ci(\p)\cdot\tdds.
\end{equation*}
This implies that $\Gt f=\Gb f$ for $f\in D(\Gt)=D(\Gb)$ which completes the proof of \eqref{eq:sim-sG-sGr}.
\end{proof}

Next we consider $\Gb$ as a domain perturbation of a simpler generator $A$. This perturbation problem is treated using a recent result stated in \autoref{thm:pert-bc-v3}.
To this end we first prove a series of lemmas. It is convenient to consider the external and internal part separately.

\smallbreak
\emph{External Part}. We introduce on $\Xe=\LpRpCl$ the operators
\begin{alignat}{3}\notag
&\Dme:=\tdds,&\qquad&D(\Dme):=\WepRpCl,\\
&\Dne:=\tdds,&\qquad&D(\Dne):=\bigl\{f\in D(\Dme):f(0)=0\bigr\},\notag
\intertext{and}\label{eq:def-Ame}
&\Aem:=\Dme\Pep-\Dme\Pem, && D(\Aem):= D(\Dme),\\
&\Ae:=\Dme\Pep-\Dne\Pem, && D(\Ae):=\bigl\{f\in D(\Dme):\Pem f(0)=0\bigr\}.\label{eq:def-A_e}
\end{alignat} 
Note that $\Dme$ and $-\Dne$  generate the strongly continuous left- and right-shift semigroups $\Tlet$ and $\Tret$ on $\Xe$, respectively, given by
\begin{alignat}{3}\label{eq:def-Telr}
&\bigl(\Tle(t)f\bigr)(\p):=f(\p+t),
\qquad
&\bigl(\Tre(t)f\bigr)(\p):=
\hat f(\p-t),
\end{alignat}
where as usual $\hat f$ denotes the extension of $f$ to $\RR$ by the value zero.
Then the following simple result holds.

\begin{lem}\label{lem:Ae-gen}
The operator $\Ae$
generates a $C_0$-semigroup $\Tet$ on $\Xe$ given by
\begin{equation}\label{eq:def-sTt-e}
\Te(t)=
\Tle(t)\Pep + \Tre(t)\Pem,\quad t\ge0.
\end{equation}
\end{lem}

Now, in the context of \autoref{sec:dom-pert}, we have $\Ae\subset\Aem$ with domain
$D(\Ae)=\ker(\Le)$ for 
\[\Le:=\Pem\delta_0:D(\Aem)\to\dXe,\]
where $\delta_0$ denotes the point evaluation in $s=0$ and  $\dXe=\rg(\Pem)$. The next result follows easily by inspection from the definition of $\Qt$ in \eqref{eq:def-Rt-Sr}.

\begin{lem}\label{lem:BCS-wp-e}
For $\tn>0$ and given $u\in\rW^{2,p}_0([0,\tn],\dXe)$ the function $x:[0,\tn]\to\Xe$,
\begin{equation*} 
x(t,\p):=\Qt u
\end{equation*}
is a classical solution of the boundary control system
\begin{equation}
\tag{BCS$^e$}\label{eq:BCSe}
\begin{cases}
\dot x(t)=\Aem x(t),&0\le t\le\tn,\\
\Le x(t)=u(t),&0\le t\le\tn,\\
x(0)=0.
\end{cases}
\end{equation}
\end{lem}

\smallbreak
\emph{Internal Part}. 
We introduce on $\Xin=\LpneCm$ the operators
\begin{alignat}{3}\label{eq:def-Am}
&\Aim:=\Cb\cdot\tdds,&\qquad&D(\Aim):=\WepneCm,\\
&\Ai:=\Cb\cdot\tdds,&&D(\Ai):=\bigl\{f\in D(\Aim):f(0)=f(1)\bigr\}\label{eq:def-Ai}
\end{alignat}
Now the following is well-known.

\begin{lem}\label{lem:Ai-gen}
The operator $\Ai$ generates the rotation $C_0$-group $\Titt$ on $\Xin$ given by
\begin{equation}\label{eq:def-Tti}
\bigl(\Ti(t)f\bigr)(\p)=
\ft(\p+\cb\cdot t),
\end{equation}
where $\ft$ denotes the $1$-periodic extension to $\RR$ of $f$ defined on $[0,1]$.
\end{lem}

As before we observe that in the context of \autoref{sec:dom-pert} we have $\Ai\subset\Aim$ with domain
\begin{equation*}
D(\Ai)=\bigl\{f\in D(\Aim):\Li f=0\bigr\}=\ker(\Li),
\end{equation*}
for
\begin{equation}\label{eq:def-Li}
\Li:=\delta_1-\delta_0:D(\Aim)\to\dXi,
\end{equation}
where $\delta_s$ denotes the point evaluation in $s\in\{0,1\}$ and  $\dXi=\CC^{m}$.

\begin{lem}\label{lem:BCS-wp}
Let $0<\tn\le\frac1{\|\cb\|_\infty}$. Then for every given $v\in\rW^{2,p}_0([0,\tn],\dXe)$ the function $x:[0,\tn]\to\Xin$ defined by
\begin{equation*} 
x(t,\p):=\bigl(\Pip\St -\Pim\Rt \bigr)v
\end{equation*}
is a classical solution of the boundary control system
\begin{equation}
\tag{BCS$^i$}\label{eq:BCS-i}
\begin{cases}
\dot x(t)=\Aim x(t),&0\le t\le\tn,\\
\Li x(t)=v(t),&0\le t\le\tn,\\
x(0)=0.
\end{cases}
\end{equation}
\end{lem}

\begin{proof} The first equation in \eqref{eq:BCS-i} follows easily by inspection.
The second equation holds since for $v\in\rW^{1,p}_0([0,\tn],\CC^m)$ we have
\begin{equation*} 
\begin{aligned}
&\delta_0(\St v)=\vh\bigl(t-\tfrac1\cbm\bigr)=0,
&\quad&\delta_1(\St v)=\vh(t-0)=v(t),
\\
&\delta_0(\Rt v)=\vh(t-0)=v(t),
&\quad&\delta_1(\Rt v)=\vh\bigl(t-\tfrac1\cbm\bigr)=0.
\end{aligned}
\end{equation*}
Finally, the third equation in \eqref{eq:BCS-i} follows from the definition of $\vh$.
\end{proof}

We are now well-prepared for the proof of our main result.

\begin{proof}[Proof of \autoref{thm:gen-Lp}]
We first assume that $c(\p,\p)$ is diagonal, i.e., $\u(\p,\p)\equiv\IId$. Then, by \autoref{lem:Gb}, $G$ is similar to $\Gb$, hence it suffices to prove that $\Gb$ given in \eqref{eq:def-sGr} generates a $C_0$-semigroup on $X$. This will be done using the results from \autoref{sec:dom-pert}.

\smallbreak
For the operators $\Aem$, $\Ae$ and $\Aim$, $\Ai$ given by \eqref{eq:def-Ame}--\eqref{eq:def-A_e} and \eqref{eq:def-Am}--\eqref{eq:def-Ai}, respectively, we define on $X=\Xe\times\Xin$ the operators
\begin{alignat*}{3}
A_m:&=\diag(\Aem,\Aim),&\quad
D(A_m):&=D(\Aem)\times D(\Aim),\\
A:&=\diag(\Ae,\Ai),&
D(A):&=D(\Ae)\times D(\Ai).
\end{alignat*}
Then $\Gb,A\subset A_m$ with domains $D(\Gb)=\ker(\Phib)=\ker(\Phie\Jp,\Phii\Jpb)$ and $D(A)=\ker(L)$ where
\begin{equation*}
L:=\diag(\Le,\Li)=
\begin{pmatrix}
\Pem\delta_0&0\\0&\delta_1-\delta_0
\end{pmatrix}:D(A_m)\to\dX=\dXe\times\dXi.
\end{equation*} 
Moreover, by \autoref{lem:Ae-gen} and \autoref{lem:Ai-gen}, $A$ generates a $C_0$-semigroup $\Tt$ on $X$ given by
\begin{equation}\label{eq:def-sTt}
T(t)=\diag\bigl(\Te(t),\Ti(t)\bigr),\quad t\ge0.
\end{equation}
Hence, the assertion follows if we verify the assumptions (i)--(iv) in \autoref{thm:pert-bc-v3} adapted to the present situation. Let $0<\tn\le\frac1{\|\cb\|_\infty}$.

\smallbreak
(i) For $t\in[0,\tn]$ and $u\in\rW^{2,p}_0([0,\tn],\dXe)$ define $\sBte u:=\Qt u\in\Xe$. Similarly, for
$v\in\rL^p([0,\tn],\dXi)$ define $\sBti v:=\bigl(\Pip\St -\Pim\Rt \bigr)v\in\Xin$ and put 
\[\sBt:=\diag\bigl(\sBte,\sBti\bigr):\dX=\dXe\times\dXi\to X=\Xe\times\Xin.\]
Then $\Bttn\subset\sL(\rL^p([0,\tn],\dX),X)$ is strongly continuous. Moreover, by \autoref{lem:BCS-wp-e} and \autoref{lem:BCS-wp}, $\sx(t):=\sBt(u,v)^\top$ solves 
\begin{equation*}
\tag{BCS} 
\begin{cases}
\dot\sx(t)=A_m\sx(t),&0\le t\le\tn,\\
L\sx(t)=\tbinom{u(t)}{v(t)},&0\le t\le\tn,\\
\sx(0)=0,
\end{cases}
\end{equation*}
hence the assertion follows. 

\smallbreak
(ii) Using the terminology introduced in \autoref{rem:C-admiss},
we have to show that $\Phib=(\Phibe,\Phibi)$ is $p$-admissible for the semigroup $\Tt$ generated by $A$. By the representations of $T(t)$ in \eqref{eq:def-sTt} and $\Te(t)$, $\Ti(t)$ in \eqref{eq:def-sTt-e},  \eqref{eq:def-Tti}, this is equivalent to  the $p$-admissibility of
\begin{enumerate}[(1)]
\item  $\Phibe$ for the semigroup $\Tet$ given in \eqref{eq:def-sTt-e}, and
\item $\Phibi$ for the group $\Titt$ given in \eqref{eq:def-Tti}. 
\end{enumerate}

To show (1), we first choose  $f\in D(\Ae)\subset\Xe=\LpRpCl$. Then similarly as in the proof of  \cite[Lem.~2.2]{EKF:19} it follows for that
\begin{align*} 
\int_0^{\tn}\bigl\|\Phibe\,\Te(t)f\bigr\|_{\dX}^p\dt
&=\int_0^{\tn}\Bigl\|\int_0^{+\infty}d\eta(s)\,\bigl(\Pep f(s+t)+\Pem\fh(s-t)\bigr)\Bigr\|_{\dX}^p\dt\\
&\le\|\eta\|^{p-1}\cdot\int_0^{+\infty}\int_0^{\tn}\bigl\|\fh(s\pm t)\bigr\|_{\CC^{\ell} }^p\dt\,d|\eta|(s)\\ 
&\le\|\eta\|^{p}\cdot\|f\|^p_{p},
\end{align*}
where $\eta$ is given by \eqref{eq:Phi-generic-e} for $\Psi=\Phibe$ and  $\fh$ denotes the extension of $f$ to $\RR$ by the value zero. This proves (1).

\smallbreak
Next, we verify (2). Let $f\in D(\Ai)\subset\Xin=\LpneCm$. Then, again as in the proof of  \cite[Lem.~2.2]{EKF:19}, it follows that
\begin{align*}
\int_0^{\tn}\bigl\|\Phibi\,\Ti(t)f\bigr\|_{\dX}^p\dt
&=\int_0^{\tn}\Bigl\|\int_0^1d\eta(s)\,\ft(s+\cb\cdot t)\Bigr\|_{\dX}^p\dt\\
&\le\|\eta\|^{p-1}\cdot\int_0^1\int_0^{\tn}\bigl\|\ft(s+\cb\cdot t)\bigr\|_{\CC^m}^p\dt\,d|\eta|(s)\\
&\le{\|\eta\|^p}\cdot{M^p}\cdot\int_0^1\bigl\| f(r)\bigr\|_{\CC^m}^p\dr\\
&\le{\|\eta\|^p}\cdot M^p\cdot\|f\|^p_{p},
\end{align*}
where $\eta$ is given by \eqref{eq:Phi-generic-e} for $\Psi=\Phibi$, $M:=\max\{|\varphi_1(1)|,\ldots,|\varphi_m(1)|\}$ and  $\ft$ denotes the $1$-periodic extension of $f$ to $\RR$.
This proves (2) and completes the proof of (ii).

\smallbreak
(iii)  By part~(i) we have 
\begin{equation*}
\sBt=
\begin{pmatrix}
\Qt&0\\
0&\Pip\St-\Pim\Rt
\end{pmatrix}
\in\sL\bigl(\rL^p\bigl([0,\tn],\dX\bigr),X\bigr),
\quad t\in[0,\tn],
\end{equation*}
where $\Qt,\St$ and $\Rt$ are defined in \eqref{eq:def-Rt-Sr}. 
This combined with \autoref{lem:bdd-ext} implies the assertion and also yields
\begin{equation}\label{eq:rep-1-F} 
\sQ_{\tn}=
\sRtn.
\end{equation}

\smallbreak
(iv) This follows immediately from \eqref{eq:rep-1-F} and the invertibility assumption on $\sR_{\tn}$.

\smallbreak
Summing up, we conclude that for $\u(\p,\p)\equiv\IId$, i.e., for $c(\p,\p)=\diag(\ce(\p),\ci(\p))$, the operator $G$ in \eqref{eq:def-G-general} generates a $C_0$-semigroup if $\sRtn$ given by \eqref{eq:def-sR} (for $\u(\p,\p)\equiv\IId$) is invertible.

\smallbreak
Now assume that $c(\p,\p)$ is given by \eqref{eq:rep-a(s)} for a Lipschitz continuous matrix function $\u(\p,\p)$. Let $\lambda(\p,\p):=\diag(\ce(\p),\ci(\p))$. Then, via the similarity transformation induced by $\u(\p,\p)$, we see that
\[
G=\u(\p,\p) \lambda(\p,\p)\u(\p,\p)^{-1}\cdot\tdds,\quad
D(G)=\ker(\Phi),
\]
is similar to the operator
\begin{equation*}
\tilde G:=
\lambda(\p,\p)\cdot\tdds+P,\quad
D(\tilde G):=\ker\bigl(\Phi\cdot\u(\p,\p)\bigr),
\end{equation*}
for
\[
P:=
\lambda(\p,\p)\u^{-1}(\p,\p)\u'(\p,\p)
\in\sL(X).
\]
Hence, by similarity and bounded perturbation $G$, is a generator if and only if $\Gt$ is. 
However, by what we proved previously for diagonal $\lambda(\p,\p)$, the matrix $\Gt$ is a generator if $\sRtn$ given by \eqref{eq:def-sR} is invertible. This completes the proof of \autoref{thm:gen-Lp}.
\end{proof}

Since an operator $G$ generates a group if and only if $\pm G$ both generate semigroups the above result gives also conditions for $G$ to be a group generator. 
In particular, if the external part is empty, i.e., if $X=\Xin=\LpneCm$, the boundary space $\dX=\CC^m$ for $\pm G$ remains invariant.
By applying \autoref{thm:gen-Lp} to the velocity matrices $\pm c(\p,\p)$  we then obtain the operators $\sRtn^{i,+}, \sRtn^{i,-}\in\sL(\rL^p([0,\tn],\CC^{m}))$ given by
\begin{equation}\label{eq:def-sR+-}
\begin{aligned}
(\sRtn^{i,+} v)(t)&=
\Phibi\cdot\bigl(\Pip\St -\Pim\Rt \bigr)v\\
(\sRtn^{i,-} v)(t) &=
\Phibi\cdot\bigl(\Pip\Rt -\Pim\St \bigr)v
\end{aligned} 
\end{equation}
for $t\in[0,\tn]$ and $v\in\rW^{1,p}_0([0,\tn],\CC^m)$. Then the following holds.

\begin{cor} \label{cor:gen-group}
Let \autoref{asu:s-asu-Lp}  be satisfied.
If there exists $\tn>0$ such that the operators $\sRtn^{i,+},\sRtn^{i,-}\in\sL(\rL^p([0,\tn],\CC^{m}))$ given by \eqref{eq:def-sR+-} are both invertible, then the operator
\begin{equation*}
\begin{aligned}
G^i&:=c(\p)\cdot\tdds,\\
D(G^i)&:=\bigl\{f\in\WepneCm: \Phii f=0\bigr\},
\end{aligned}
\end{equation*}
generates a $C_0$-group on $\Xin=\LpneCm$.
\end{cor}

\subsection{A Necessary Condition}\label{sec:necessary} 

In this subsection we will deduce a necessary  condition for $G$ defined in \eqref{eq:def-G-general} to generate a $C_0$-semigroup on $X=\LpRpCl\times\LpneCm$. In the next section we then apply this result to particular sets of boundary conditions and show that in many important cases the invertibility condition of the operator $\sRtn$ in \autoref{thm:gen-Lp} is also necessary for $G$ to be a generator.

In order to do so we first note that by \autoref{lem:Gb} and the final part of the proof of \autoref{thm:gen-Lp} the operator $G$ with domain $D(G)=\ker(\Phi)$ is, up to a bounded perturbation, similar to $\Gb$ with domain $D(\Gb)=\ker(\Phib)$ given by \eqref{eq:def-sGr}
for the boundary operator
\begin{equation}\label{eq:def-Phiq}
\Phib:=(\Phie,\Phii)\cdot\u(\p,\p)\cdot\diag(\Jp,\Jpb)
\in\sL(\CnRpCl\times\CneCm,\dX).
\end{equation}
We then consider $\Gb$ as a boundary perturbation of
\begin{equation}\label{eq:def-A}
A:=\diag\Bigl(\bigl(\Pep-\Pem\bigr)\cdot\tdds,\Cb\cdot\tdds\Bigr),
\quad
D(A):=\ker(L)
\end{equation}
for the boundary operator 
\begin{equation}\label{eq:def-L}
L:=\diag\bigl(\Pem\delta_0,\Pip\delta_1+\Pim\delta_0\bigr)\in\sL\bigl(\CnRpCl\times\CneCm,\dX\bigr).
\end{equation}
Since the generator property is invariant under similarity transformations and bounded perturbations, \autoref{thm:NC-gen} will give the following result.

\begin{thm}\label{thm:NC-1}
Let the operators $A$ and $G$ be given by \eqref{eq:def-A} and \eqref{eq:def-G-general}, respectively. Moreover, assume that the boundary operator $\Phib$ in \eqref{eq:def-Phiq} can be decomposed as
\begin{equation*}
\Phib=\Phi_0+B+V\cdot L
\end{equation*}
for some $\Phi_0\in\sL(\CnRpCl\times\CneCm,\dX)$, $\B\in\sL(X,\dX)$, a non-invertible $V\in\sL(\dX)$, and $L$ given by \eqref{eq:def-L}.
If $G$ generates a $C_0$-semigroups on $X$ then
\begin{equation}\label{eq:est-PnLl}
\lim_{\lambda\to+\infty}\lambda^{\frac1p}\cdot\|\Phi_0 L_\lambda\|=+\infty
\end{equation}
where $L_\lambda\in\sL(\dX,X)$ is defined by $L_\lambda x:=\epsl(\p,\p)\cdot x$ for
\begin{equation*}
\epsl(r,s):= \diag\bigl(e^{-\lambda r},
{\Pip\cdot e^{\lambda(s-1)\cdot\Cb^{-1}}+\Pim\cdot e^{\lambda s\cdot\Cb^{-1}}}\bigr),
\quad r\in\RR_+,\;s\in[0,1]
\end{equation*}
 and $\Cb\in\rM_m(\RR)$ given by \eqref{eq:def-Cb}.
\end{thm}

\begin{proof} First assume that $B=0$. To verify the hypotheses of \autoref{thm:NC-gen} we note that $A,\Gb\subset A_m$ for 
\begin{equation}\label{eq:def-Am2}
A_m:=\diag\Bigl(\bigl(\Pep-\Pem\bigr)\cdot\tdds,\Cb\cdot\tdds\Bigr),
\quad
D(A_m):=\WepRpCl\times\WepneCm.
\end{equation}
Moreover, $A$ generates a diagonal shift-semigroup on $X$, shifting to the left on the ranges of $P^{e,i} _+$ and to the right on the ranges of $P^{e,i} _-$. Since $\ker(\lambda-A_m)=\{\epsl\cdot x:x\in\dX\}$ and $L\epsl\cdot x=x$ for all $\lambda>0$ and $x\in\dX$, it follows that $L_\lambda=(L|_{\ker(\lambda-A_m)})^{-1}$, i.e., it coincides indeed with the Dirichlet operator introduced in \autoref{lem:Ll}. 

Next, we estimate $\|\Ll x\|_p^p$ for $\lambda>0$ and $x=(x^e,x^i)=(x^e,(x^i_k)_1^m)\in\dX = \rg(\Pem)\times\CC^m$. In fact, since
\begin{equation*}
\lim_{\lambda\to+\infty}e^{\frac{-p\lambda}{\cb_k}}=0\text{ if }\cb_k>0
\quad\text{and}\quad
\lim_{\lambda\to+\infty}e^{\frac{p\lambda}{\cb_k}}=0\text{ if }\cb_k<0,
\end{equation*}
there exists $\lambda_0>0$ such that
for $\cb_{\mathrm{min}}:=\min\{|\cb_k|:k=1,\ldots,m\}>0$ and $\lambda\ge\lambda_0$ it holds
\begin{align}\label{eq:est-Ll}
\|\Ll x\|_p^p&=\tfrac1{p\lambda}\cdot\Bigl(\|x^e\|^p+
\sum_{\cb_k>0}\cb_k\cdot\bigl(1-e^{\frac{-p\lambda}{\cb_k}}\bigr)\cdot|x^i_k|^p+
\sum_{\cb_k<0}|\cb_k|\cdot\bigl(1-e^{\frac{p\lambda}{\cb_k}}\bigr)\cdot|x^i_k|^p
\Bigr)\\
&\ge\tfrac{1}{p\lambda} \left(\|x^e\|^p+\cb_{\mathrm{min}}\cdot \|x^i\|^p \right)
\notag.
 \end{align}
  This implies that there exist $M>0$ and $\lambda_0>0$ such that
\begin{equation*}
\|\Ll x\|\ge M^{-1}\cdot\lambda^{-\frac1p}\cdot\|x\|\quad\text{for all }x\in\dX,\;\lambda\ge\lambda_0.
\end{equation*}
Using this estimate we conclude that for all $0\ne f=\Ll Lf\in\ker(\lambda-A_m)$ we have
\begin{align*}
\|Lf\|&=\|Lf\|\cdot\frac{\|f\|}{\|\Ll Lf\|}\le\|Lf\|\cdot\frac{\|f\|}{M^{-1}\cdot\lambda^{-\frac1p}\cdot\|Lf\|}
= M\cdot\lambda^{\frac1p}\cdot\|f\|,
\end{align*}
i.e., $\|L\|_\lambda\le M\cdot\lambda^{\frac1p}$. The assertion now follows immediately from \autoref{thm:NC-gen}. 

Finaly, take $B\in\sL(X,\dX)$. Since, by \eqref{eq:est-Ll}  there exists $\tilde M>0$ such that $\|\Ll\|\le\tilde M\cdot\lambda^{-\frac1p}$ for all $\lambda>0$, $\lambda^{\frac1p}\cdot\|B\Ll\|$ remains bounded as $\lambda\to+\infty$. By replacing $\Phin$ by $\Phin+B$, by the first part of the proof this implies the assertion for arbitrary $B$.
\end{proof}

At a first glance this result might seem rather artificial, however, as we will see in the next subsection, it significantly simplifies (like \autoref{thm:gen-Lp}) in case the boundary operator $\Phi$ is mainly given by point evaluations in $r=0\in\RRp$ and $s=0,1\in[0,1]$.

\begin{rem}
In \autoref{thm:NC-1} we used in the internal part the boundary operator $L^i=\Pip\delta_1+\Pim\delta_0$ for the unperturbed generator $\Ai=\Cb\cdot\dds$, cf. \eqref{eq:def-L}, which generates nilpotent left/right-shift semigroups. This choice is different from \eqref{eq:def-Li} in the proof of \autoref{thm:gen-Lp} which was associated to left/right-rotation groups. We note that in \eqref{eq:def-L} one could also choose 
\begin{equation}\label{eq:def-L1}
L=\diag\bigl(\Pem\delta_0,\delta_1-\delta_0\bigr).
\end{equation}
By similar computations as above it follows again that $\|L\|_\lambda\ge M^{-1}\cdot\lambda^{-\frac1p}$ for a suitable constant $M>0$ while $\Ll$ in this case is given by $\Ll x=\epsl\cdot x$ for
\begin{equation}\label{eq:def-L2}
\epsl(r,s):= \diag\bigl(e^{-\lambda r},
\Pip\cdot e^{\lambda(s-1)\cdot\Cb^{-1}}\cdot\bigl(Id-e^{-\lambda\cdot\Cb^{-1}}\bigr)-\Pim\cdot e^{\lambda s\cdot\Cb^{-1}}\cdot\bigl(Id-e^{\lambda\cdot\Cb^{-1}}\bigr)\bigr).
\end{equation}
Summing up,  \autoref{thm:NC-1} remains true if the operators $L$ and $\Ll$ are replaced by their counterparts defined by \eqref{eq:def-L1} and \eqref{eq:def-L2}.
\end{rem}

\subsection{Reformulations of the Boundary Conditions}\label{subcsec:reformulation}

Next, we study the invertibility of $\sRtn$ appearing in \autoref{thm:gen-Lp}. To this end we assume that as in \eqref{eq:Phi-generic-e} the boundary operators are given by
\begin{align*}
\Phie h&=\int_0^{+\infty}\kern-7pt d\eta^e(r)\, h(r)
\quad\text{for all}\quad
h\in\LpRpCl,\\
\Phii h&=\int_0^1 d\eta^i(s)\, h(s)
\quad\text{for all}\quad
h\in\LpneCm,
\end{align*} 
for functions $\eta^e:\RRp\to\sL(\CC^\ell,\dX)$ and $\eta^i:[0,1]\to\sL(\CC^m,\dX)$ of bounded variation. Using these representations we define the operators
\begin{alignat}{3}\label{eq:def-Vne}
\Vne:&=\lim_{r\to0^+}\eta^e(r)-\eta^e(0)=\int_{\{0\}}d\eta^e(r)=\nu^e(\{0\})\in\sL(\CC^\ell,\dX),\\\label{eq:def-Vni}
\Vni:&=\lim_{s\to0^+}\eta^i(s)-\eta^i(0)\int_{\{0\}}d\eta^i(s)=\nu^i(\{0\})\in\sL(\CC^m,\dX),\\\label{eq:def-Vei}
\Vei:&=\eta^i(1)-\lim_{s\to1^-}\eta^i(s)=\int_{\{1\}}d\eta^i(s)=\nu^i(\{1\})\in\sL(\CC^m,\dX),
\end{alignat}
where $\nu^e$ and $\nu^i$ denote the vector measures on $\RRp$ and $[0,1]$ associated to the functions $\eta^e$ and $\eta^i$, respectively.

\smallskip
As we will see, the invertibility of $\sRtn$ can be characterized in terms of these three finite dimensional operators. Combined with \autoref{thm:gen-Lp} this gives a simple condition implying the generator property of the operator $G$ in \eqref{eq:def-G-general}. In particular, it follows that the well-posedness of \eqref{eq:acp} for $G$ as in \eqref{eq:def-G-general} only depends on the boundary condition localized to the endpoints of the ``edges'' $\RRp$ of the external part and $[0,1]$ of the internal part.
Recall that $\dX=\dXe\times\dXi=\rg(\Pem)\times\CC^m\subseteq\CC^\ell\times\CC^m$, cf.~\eqref{eq:def-dX}.

\begin{lem}\label{lem:inv-R0}
Let $\Vne,\Vni,\Vei$ be defined by \eqref{eq:def-Vne}--\eqref{eq:def-Vei} and let $\tmax:=\frac1{\|\cb\|_\infty}$. Then the following conditions are equivalent.
\begin{enumerate}[(a)]
\item The operator
\[ 
\sR_0:=\bigl(\Vne\qe(0),\Vei\qi(1)\Pip-\Vni\qi(0)\Pim\bigr)\in\sL(\dX)
\]
is invertible.
\item There exists $0<\tn\le\tmax$ such that $\sRtn\in\sL(\rL^p([0,\tn],\dX))$ given by \eqref{eq:def-sR} is invertible.
\end{enumerate}
\end{lem}

\begin{proof} Since in this proof we consider the ``input-output map'' $\sRtn$ for different boundary operators we use for $\Psi=(\Psi^e,\Psi^i)\in\sL(\CnRpCl\times\CneCm,\dX)$ the notations
\begin{align*}
\bar\Psi&:=(\bar\Psi^e,\bar\Psi^i):=\bigl(\Psi^e\cdot\qe(\p)\cdot\Jp,\Psi^i\cdot\qi(\p)\cdot\Jpb\bigr)\in\sL\bigl(\CnRpCl\times\CneCm,\dX\bigr)
\intertext{and}
\sRtn^\Psi&:=
\bigl(\bar\Psi^e\cdot\Qt,\bar\Psi^i\cdot\bigl(\Pip\cdot\psi -\Pim\bigr)\cdot\Rt\bigr)\in\sL(\rL^p([0,\tn],\dX)),
\end{align*}
where $\psi(s):=1-s$, cf.~\eqref{eq:def-sR}, \eqref{eq:Bti-psi}. Observe that then $\sR_{\tn}^{\Phi}=\sRtn$.
Let
\begin{equation*}
\Phin=(\Phien,\Phiin)
:=\bigl(\Vne\delta_0,\Vei\delta_1+\Vni\delta_0\bigr).
\end{equation*}
Then, a simple computation shows that $\sR_{\tn}^{\Phin}=\sR_0$
 for every $\tn>0$. Next, we show that
\begin{equation}\label{eq:sRt->0}
\bigl\|\sR^{\Phi-\Phin}_{\tn}\bigr\|\to0\quad\text{as }\tn\to0^+.
\end{equation}
For this it suffices to verify that $\|U^\Psi_{\tn}\|\to0$ for
\begin{equation}\label{eq:Psi-e-i}
\Psi=
\begin{cases}
\bigl(\Phibe-\bar\Phi^e_0\bigr),&\text{in the external case},\\
\bigl(\Phibi-\bar\Phi^i_0\bigr)\cdot\bigl(\Pip\cdot\psi -\Pim\bigr),&\text{in the internal case},
\end{cases}
\end{equation}
where $U^\Psi_{\tn}$ is defined by \eqref{eq:def-Utn}. However, by \eqref{eq:est-Utn} we have $\|U^\Psi_{\tn}\|\le|\eta|(I_{\tn})$ for $\eta$ as in \eqref{eq:Phi-generic-e} and $I_{\tn}$ given by \eqref{eq:def-I_tn}. Since by \eqref{eq:def-Vne}--\eqref{eq:def-Vei} in both cases, the external and the internal one,  $|\eta|(\{0\})=0$ we infer from $I_{t_1}\subset I_{\tn}$ for $0<t_1<\tn\le\tmax$ and
\[
\bigcap_{0<\tn\le t_{\max}} I_{\tn}=\{0\}
\]
that
\[
0=|\eta|\bigl(\{0\}\bigr)=\lim_{\tn\to0^+}|\eta|\bigl(I_{\tn}\bigr)
\]
which implies \eqref{eq:sRt->0}.
Finally, observe that
\begin{equation}\label{eq:dec-sR}
\sR^\Phi_{\tn}=\sR^{\Phin}_{\tn}+\sR^{\Phi-\Phin}_{\tn}=\sR_0+\sR^{\Phi-\Phin}_{\tn}
\end{equation}

\smallbreak
(a)$\Rightarrow$(b). By \eqref{eq:dec-sR} we have
\begin{equation}\label{eq:rep-sRP}
\sR^\Phi_{\tn}=\sR_0\cdot\bigl(\Id+\sR_0^{-1}\cdot\sR^{\Phi-\Phin}_{\tn}\bigl).
\end{equation}
Since $\|\sR_0^{-1}\|_{\sL(\rL^p([0,\tn],\dX))}$ is independent of $\tn>0$, the representation \eqref{eq:rep-sRP} combined with \eqref{eq:sRt->0} implies that $\sR_{\tn}^\Phi$ is invertible for $\tn>0$ sufficiently small.

\smallbreak
(b)$\Rightarrow$(a). We first observe that by causality we have $\sR_{\tn}^\Phi|_{[0,t]}=\sR_t^\Phi$ and
 $(\sR_{\tn}^\Phi)^{-1}|_{[0,\tn]}=(\sR_{t}^\Phi)^{-1}$ for all $t\in(0,\tn]$.
Hence, $\sR_{t}^\Phi$ is invertible and $\|(\sR_{t}^\Phi)^{-1}\|_{\sL(\rL^p([0,t],\dX))}\le\|(\sR_{\tn}^\Phi)^{-1}\|_{\sL(\rL^p([0,\tn],\dX))}$ for all $t\in(0,\tn]$. Combined with \eqref{eq:dec-sR} (with $\tn$ replaced by $t$) this implies
\begin{equation}\label{eq:rep-sR0}
\sR_0=\sR_{t}^\Phi\cdot\bigl(\Id-(\sR_{t}^\Phi)^{-1}\cdot\sR^{\Phi-\Phin}_{t}\bigr).
\end{equation}
Since $\|(\sR_{t}^\Phi)^{-1}\|_{\sL(\rL^p([0,t],\dX))}$ remains bounded for $t\to0^+$, the representation \eqref{eq:rep-sR0} together with \eqref{eq:sRt->0} (again with $\tn$ replaced by $t$) implies that $\sR_0$ considered as an operator on $\rL^p([0,t],\dX)$ is invertible for $t>0$ sufficiently small. Clearly, this implies that $\sR_0\in\sL(\dX)$ is invertible as well.
\end{proof}

\begin{rem}\label{rem:inv-BP}
Let $\B\in\sL(X,\dX)$. Since $\dX$ is finite dimensional there exists $k=(k^e,k^i)\in\rL^q(\RRp,\sL(\CC^l,\dX))\times\rL^q([0,1],\sL(\CC^m,\dX))$ for $\frac1p+\frac1q=1$ such that for $f=(\fe,\fin)\in X$ we have
\[
\B f=\int_0^\infty k^e(r)\cdot f^e(r)\dr+\int_0^1 k^i(s)\cdot f^i(s)\ds.
\]
Hence, the matrices $\Vne,\Vni,\Vei$ defined in \eqref{eq:def-Vne}--\eqref{eq:def-Vei} remain invariant if we change the boundary operator $\Phi$ into $\Phi+\B$. By the previous lemma this implies that the invertibility of $\sRtn$ for sufficiently small $\tn>0$ is invariant under bounded perturbations of the boundary operator $\Phi$. 
\end{rem}

\autoref{lem:inv-R0} together with \autoref{thm:gen-Lp} gives the following.

\begin{cor}\label{cor:gen-sR0}
Let $\Vne,\Vni,\Vei$ be given by \eqref{eq:def-Vne}--\eqref{eq:def-Vei}. If the operator
\[ 
\sR_0=\bigl(\Vne\qe(0),\Vei\qi(1)\Pip-\Vni\qi(0)\Pim\bigr)\in\sL(\dX)
\]
is invertible, then the operator $G$ defined in \eqref{eq:def-G-general} generates a $C_0$-semigroup on the space $X=\LpRpCl\times\LpneCm$. 
\end{cor}

Denote by  $P_\pm^{c(\p,\p)}$ the spectral projections of $c(\p,\p)$ associated to the positive/negative eigenvalues of $c(\p,\p)$. Then $\u(\p,\p)P_\pm=P_\pm^{c(\p,\p)}\u(\p,\p)$ where the columns of $q(\p,\p)$ contain linearly independent eigenvectors of $c(\p,\p)$. Combining this with \autoref{cor:gen-sR0}
and the fact that a linear map on a finite dimensional vector space is surjective if and only if it is invertible immediately implies the following non-trivial generalization of the main result Thm.~1.5 in \cite{JMZ:15} to possibly non-local boundary conditions and non-compact metric graphs. 
At the same time it generalizes the main result Thm.~4.10 in \cite{JW:19} where only equations on semi-axis are considered. 

\begin{cor}
Let $Z_-^{e}(0)\subseteq\CC^\ell$ be the eigenspace of $\ce(0)$ corresponding to its negative eigenvalues. Similarly, for $s=0,1$ let $Z_\pm^{i}(s)\subseteq\CC^m$ denote the eigenspaces of $\ci(s)$ corresponding to its positive/negative eigenvalues. 
Finally, let $\Vne,\Vni,\Vei$ be given by \eqref{eq:def-Vne}--\eqref{eq:def-Vei}. 
If
\[
\Vne Z^e_-(0)+\Vei Z^i_+(1)+\Vni Z^i_-(0)=\dX
\]
then $G$ defined in \eqref{eq:def-G-general} generates a $C_0$-semigroup on $X=\LpRpCl\times\LpneCm$.
\end{cor}

Next, we consider the case where the boundary conditions are given via ``boundary matrices''. Recall that $\dX=\rg(\Pem)\times\CC^m\subseteq\CC^{\ell + m }$
 and denote by 
\[\n:= \dim \dX = \rank(\Pem) + m.\]

\begin{cor}\label{cor:gen-matrix}
Let $\Vne\in M_{\n\times\ell}(\CC)$, $\Vni,\Vei\in M_{\n\times m}(\CC)$ and $\B\in\sL(X,\dX)$. Then
\begin{equation*}
\begin{aligned}
G&=c(\p,\p)\cdot\tdds,\\
D(G)&=\left\{f=\tbinom{\fe}{\fin}\in\WepRpCl\times\WepneCm:\Vne\fe(0)+\Vni\fin(0)-\Vei\fin(1)=\B f
\right\},
\end{aligned}
\end{equation*}
generates a $C_0$-semigroup on $X=\LpRpCl\times\LpneCm$ if and only if the operator
\begin{equation}\label{eq:R0-matrix}
\sR_0:=\bigl(\Vne\qe(0),\Vei\qi(1)\Pip-\Vni\qi(0)\Pim\bigr)\in\sL(\dX)
\end{equation}
is invertible.
\end{cor}

\begin{proof}
The ``if'' part follows from \autoref{rem:inv-BP} and \autoref{cor:gen-sR0} by choosing $\Phie:=  \Vne\delta_0$ and $\Phii:=\Vni\delta_0-\Vei\delta_1$ which implies $\nu^e(\{0\})=\Vne$, $\nu^i(\{0\})=\Vni$, $\nu^i(\{1\})=\Vei$ in \eqref{eq:def-Vne}--\eqref{eq:def-Vei}. 

For the ``only if'' part we assume that $\sR_0$ is not invertible and have to show that in this case $G$ is not a generator.
To this end we use \autoref{thm:NC-1} where, without loss of generality, we assume that $\B=0$. Observe, that as $\sR_0$ also 
\[
V:=\sR_0\cdot\diag(\Id_{\dXe},-\Id_{\dXi})=\bigl(\Vne\qe(0),\Vni\qi(0)\Pim-\Vei\qi(1)\Pip\bigr)\in\sL(\dX)
\]
is not invertible. Next we decompose 
\[\Phib=\bigl(\Vne\qe(0)\delta_0,\Vni\qi(0)\delta_0-\Vei\qi(1)\delta_1\bigr)=\Phin+V\cdot L\]
where $L$ is given by \eqref{eq:def-L}. Hence,
\begin{equation*}
\Phin=\Phib-V\cdot L=\bigl(\Vne\qe(0)\Pep\delta_0,\Vni\qi(0)\Pip\delta_0-\Vei\qi(1)\Pim\delta_1\bigr).
\end{equation*}
This yields, independently on the generator property of $G$, that
\begin{align*}
\|\Phin\Ll\|&=\bigl\|\bigl(0,\Vni\qi(0)\cdot\Pip e^{-\lambda\Cb^{-1}}-\Vei\qi(1)\cdot\Pim e^{\lambda\Cb^{-1}}\bigr)\bigr\|\\
&\le M\cdot e^{-\cb_{\mathrm{min}}\cdot\lambda}
\end{align*}
for a suitable constant $M>0$, $\cb_{\mathrm{min}}:=\min\{|\cb_k|:k=1,\ldots,m\}>0$ and $\lambda>0$. Thus,
\begin{equation*}
\lim_{\lambda\to+\infty}\lambda^{\frac1p}\cdot\|\Phi_0 L_\lambda\|=0
\end{equation*}
and by \eqref{eq:est-PnLl} in \autoref{thm:NC-1} the operator $G$ cannot be the generator of a $C_0$-semigroup. This completes the proof.
\end{proof}

In the compact case 
$X=\Xin=\LpneCm$ we obtain a simple determinant condition for the (semi-)group generation.

\begin{cor}\label{cor:gen-matrix-det}
Let $V_0,V_1\in\rM_m(\CC)$ and $\B\in\sL(\LpneCm,\CC^m)$.
Then the operator
\begin{equation*}
\begin{aligned}
G&=c(\p)\cdot\tdds,\\
D(G)&=\bigl\{f\in\WepneCm:V_0f(0)-V_1f(1)=\B f
\bigr\},
\end{aligned}
\end{equation*}
generates a $C_0$-semigroup on $\LpneCm$ if and only if
\[\det\bigl(V_1q(1) P_+ -  V_0 q(0) P_-\bigr)\ne0,\]
Moreover, $(G,D(G))$ generates a  $C_0$-group if and only if in addition
\[\det\bigl(V_1q(1) P_- -  V_0 q(0) P_+\bigr)\ne0.\]
\end{cor}

\begin{rem}
Observe that in the group case one can combine the two determinant conditions to the single condition
\[\det\left(V_1q(1) P_+ q(0)^{\top} V_0^{\top} +  V_0 q(0) P_-q(1)^{\top} V_1^{\top}\right)\ne0.\]
Moreover, since $q(\p)$ is invertible, in the case when all the velocities are of the same sign, the conditions simplify to $\det V_0 \ne 0$, $\det V_1 \ne 0$, and $\det(V_0 V_1^{\top}) \ne 0$, respectively.
\end{rem} 

We conclude the section with a few basic examples. 

\begin{exa} 
Consider a simple transport operator $G\subset \frac{d}{ds}$ on $\rL^p[0,1]$ with  non-local boundary conditions that appear, for example, in the McKendrick--von Foerster equation, cf.~\cite{Has:91}. 
More precisely, for $h\in\rL^q[0,1]$ where $\frac1p+\frac1q=1$, we define the domain
\[D(G):=\left\{f\in\rW^{1,p}\bigl([0,1],\CC\bigr) \,\bigm| \, f(1) = \int_0^1 h(s)f(s) \ds\right\}.\]
In our setting this corresponds to $\ell=0$,  $m=1$, the velocity coefficient $c(\p)\equiv1$, the state space $X= \rL^{p}[0,1]$, the boundary space $\dX=\CC$, the boundary ``matrices'' $V_0=0$, $V_1=-1$, the projections $P_+=1$, $P_-=0$ and the bounded boundary functional $\B\in X'$ given by
$\B f:=\int_0^1 h(s)f(s) \ds$.
By \autoref{cor:gen-matrix-det} this implies 
that $G$ generates a $C_0$-semigroup but not a group on $\rL^p[0,1]$.
\end{exa}

\begin{exa}\label{ex:loop}
Consider a system of transport equations on two closed unit intervals, so $m=2$ and $\ell=0$. More precisely, we take the operator
\begin{equation*}
\begin{aligned}
G&={q(\p)\cdot  \begin{pmatrix} \lambda_1(\p)  & 0 \\ 0 &  \lambda_2 (\p) \end{pmatrix}\cdot q(\p)^{-1} \cdot\tdds, } \quad
D(G)&=\bigl\{f\in \rW^{1,p}\bigl([0,1], \CC^2\bigr) :V_0f(0)=V_1f(1)\bigr\},
\end{aligned}
\end{equation*}
for functions $\lambda_1(\p), \lambda_2(\p)\in\rL^\infty([0,1],\CC)$ of constant sign (cf.~\eqref{eq:aek-beschr}), Lipschitz continuous matrix-valued function $q(\p)$ such that $q(\p),q^{-1}(\p)\in\rL^\infty([0,1],\rM_2(\CC))$, and complex $2\times 2 $ matrices
\begin{equation}\label{eq:loop-bc}
V_0 = \begin{pmatrix} a & b\\ c& d\end{pmatrix}\quad\text{and}\quad V_1 = \begin{pmatrix} \alpha & \beta\\ \gamma& \delta\end{pmatrix}.
\end{equation} 
Let us first examine the simplest case when $q(\p)=Id$.
If $\lambda_1(\p),\lambda_2(\p)$ are both positive (resp.~negative), \autoref{cor:gen-matrix-det} yields the generation of a $C_0$-semigroup  if and only if $\det V_1 \ne 0$ (resp.~$\det V_0 \ne 0$) and the generation of a group if and only if both matrices are invertible. In case $\lambda_1(\p)>0>\lambda_2(\p)$ (resp.~$\lambda_1(\p)<0<\lambda_2(\p)$), the generation condition for the $C_0$-semigroup becomes
\begin{equation}\label{eq:det-ab}
\det\begin{pmatrix} \alpha & -b\\ \gamma & -d\end{pmatrix} \ne 0 \quad (\text{resp.~} \det\begin{pmatrix} \beta & -a\\ \delta & -c\end{pmatrix} \ne 0).
\end{equation}
Moreover, if both determinants are nonzero we obtain a $C_0$-group.

Next, we consider general $q(\p)$. Observe, that in the case when $\lambda_1(\p),\lambda_2(\p)$ are both of the same sign, by invertibility,  $q(\p)$  has  no influence to the generation condition. However, if   $\lambda_1(\p)$ and $\lambda_2(\p)$ have the opposite signs, the entries of $q(0)$ and $q(1)$ may significantly contribute to the invertibility of the matrices in question. To demonstrate this,  take $\lambda_1(\p)>0>\lambda_2(\p)$ and  $V_0=V_1= Id$. For $q(\p)=Id$ both conditions in \eqref{eq:det-ab} are fullfiled and  yield a $C_0$-group. On the other hand, if we take 
\[ q(s):= \begin{pmatrix} 2-s & s-1\\1-s & s \end{pmatrix},\]
which satisfies \autoref{asu:s-asu-Lp}, the matrix $V_1q(1) P_+ -  V_0 q(0) P_-$ is singular. 
\end{exa}

\begin{exa}\label{ex:halfline}
Consider a system of transport equations on two different intervals: the first equations takes place on $\RRp$ and the second one on $[0,1]$. We thus have $\ell=1=m$ and
\begin{equation*}
\begin{aligned}
G&=\begin{pmatrix} \ce(\p) \cdot\tdds & 0 \\ 0 &  \ci (\p) \cdot\tdds\end{pmatrix},\\
D(G)&=\left\{f=\tbinom{\fe}{\fin}\in \rW^{1,p}\bigl(\RRp, \CC\bigr) \times  \rW^{1,p}\bigl([0,1], \CC\bigr):\:\Vne\fe(0)+\Vni\fin(0)=\Vei\fin(1)
\right\},
\end{aligned}
\end{equation*}
for some velocity functions $\ce(\p)\in\rL^\infty(\RRp,\CC)$, $\ci(\p)\in\rL^\infty([0,1],\CC)$ and boundary matrices $\Vne\in M_{\n\times 1}(\CC)$ and $\Vni,\Vei\in M_{\n\times 1}(\CC)$. The size of these matrices depends on the sign of $\ce(\p)$: if  $\ce(\p) >0$ then $\n=1$ and all the matrices are scalars, i.e., we only have one boundary condition. In this case, \autoref{cor:gen-matrix}
characterizes the generation of a $C_0$-semigroup  through the condition $\Vni\ne 0$ if $\ci(\p)<0$, respectively $\Vei\ne 0$ if $\ci(\p)>0$. However, if $\ce(\p) <0$,  the boundary space $\dX$ gets two-dimensional, that is $\n=2$, and the generation condition in case of $\ci(\p)>0$ (resp.~$\ci(\p)<0$) becomes 
\[\det\begin{pmatrix} \alpha & c\\ \beta & d\end{pmatrix} \ne 0 \quad (\text{resp.~} \det\begin{pmatrix} \alpha & -a\\ \beta& -b\end{pmatrix} \ne 0),\]
where $\Vne={\alpha  \choose \beta}$, $\Vni={a \choose b}$ and $\Vei={c \choose d}$.
\end{exa}

\section{Application to Flows in Networks}\label{sec:networks}

In this section we use our abstract results to show the well-posedness of transport equations on networks. We will consider initial-boundary value problems  as in \eqref{eq:acp} where the structure of the network will be encoded in the domain $D(G)$. 

\smallskip
Consider a finite metric graph (network)  with $n$ vertices $\mv_1,\dots, \mv_n$,  $m$ \emph{internal} edges $\mei_1,\dots,\mei_m$, which we parametrize on the unit interval $[0,1]$, and $\ell$ \emph{external} edges $\mee_1,\dots,\mee_{\ell}$, parametrized on the half-line $\RR_+$. A graph without external edges is called \emph{compact}.  
For simplicity, the graph is assumed to be connected and without vertices of degree 1. 

Along the edges of this metric graph some material is transported. This flow process is described by a system of linear transport equations
\begin{equation}\label{eq:net-dif}
\begin{aligned}
\frac{d}{dt}\, u^e_k(t,s) &= \cek(s)\cdot \frac{d}{ds}\, u^e_k(t,s),&& t\ge 0,\ s>0, && k=1,\dots,\ell,\\
\frac{d}{dt}\, u^i_j(t,s) &= \cij(s)\cdot \frac{d}{ds}\, u^i_j(t,s),&& t\ge 0,\  s\in(0,1),&& j=1,\dots,m ,
\end{aligned}
\end{equation}
 for some uniformly bounded functions $\cek(\p) \in\rL^\infty(\RRp,\CC)$, $\cij(\p)\in\rL^\infty([0,1],\CC$. 
  Here,  $u^e_k$, $u^i_j$ represent the flow density and $\cek,\cij$ the flow velocity on the edge $\mee_k, \mei_j$, respectively. 
 In the vertices of the graph the material gets redistributed according to some prescribed rules which translate to boundary conditions of our problem. We shall assume that these conditions respect the network structure, that is only the values at common endpoints can interfere. 
 
While the underlying combinatorial graph is a priori not directed, the parametrization orients the edges of the corresponding metric graph.
The direction of the flow on a given edge is determined by the sign of the appropriate velocity coefficient: a negative coefficient yields a flow  from the endpoint $0$ to the endpoint $1$ on the corresponding edge and the opposite holds in the case of a positive coefficient. For a directed graph we always assume that the direction of the edges corresponds to the direction of the flow. 
Let us start with some simple examples.

\begin{exa}\label{ex:loop-graph}
Reconsider the system presented in \autoref{ex:loop} {(with $q(\p)=Id$)} and imagine now the two intervals glued together at both endpoints,  forming a double edge, so $m=2$ and $\ell = 0$. Here, we have two possibilities for the flow direction, yielding either a sink or a loop, 
as depicted in the first two pictures of \autoref{fig:loop}.  They are obtained by the appropriate combinations of the parametrization  of the intervals and the sign of the velocity coefficients. Up to isomorphism, we may assume that  $\lambda_1(\p),\lambda_2(\p)$ are both negative, and we glue the same endpoints (0-0 and 1-1) in the first case  and   mixed  endpoints  (0-1 and 1-0) in the second case.

\begin{figure}[!h]  
  \begin{subfigure}[b]{0.3\textwidth}
     \centering
      \begin{tikzpicture}[-,auto,node distance=3cm,
  thick,main node/.style={circle,fill=black,draw}]

  \node[main node] (1) at (0,0) {};
  \node[main node] (2) at (2,0) {};

\tikzset{edge/.style = {->,> = latex'}}
 \draw[edge] (1) edge  [bend left]  node [above] {$\me_1$}  (2);
 \draw[edge] (1) edge  [bend right]  node [below] {$\me_2$}  (2); 
\end{tikzpicture}
\caption{Sink}
    \end{subfigure}
      \hfill
 \begin{subfigure}[b]{0.3\textwidth}
    \centering
    \begin{tikzpicture}[-,auto,node distance=3cm,
  thick,main node/.style={circle,fill=black,draw}]

  \node[main node] (1) at (0,0) {};
  \node[main node] (2) at (2,0) {};

\tikzset{edge/.style = {->,> = latex'}}
 \draw[edge] (1) edge  [bend left]  node [above] {$\me_1$}  (2);
 \draw[edge] (2) edge  [bend left]  node {$\me_2$}  (1); 
\end{tikzpicture}
\caption{Loop}
  \end{subfigure}
   \hfill
 \begin{subfigure}[b]{0.3\textwidth}
    \centering
    \begin{tikzpicture}[-,auto,node distance=3cm,
  thick,main node/.style={circle,fill=black,draw}]

  \node[main node] (1) at (0,0) {};
  \node[main node] (2) at (2.5,0) {};

\tikzset{edge/.style = {->,> = latex'}}
 \draw[edge] (1) edge  [bend left]  node [above] {$\me_1$}  (2);
 \draw[edge] (1) edge node {$\me_2$}  (2); 
  \draw[edge] (2) edge  [bend left]  node   {$\me_3$}  (1); 

\end{tikzpicture}
\caption{Pumpkin}
 
    \end{subfigure}
 \caption{Compact graphs from \autoref{ex:loop-graph}. The edges are oriented according to the  direction of the flow.}  
  \label{fig:loop}
\end{figure}
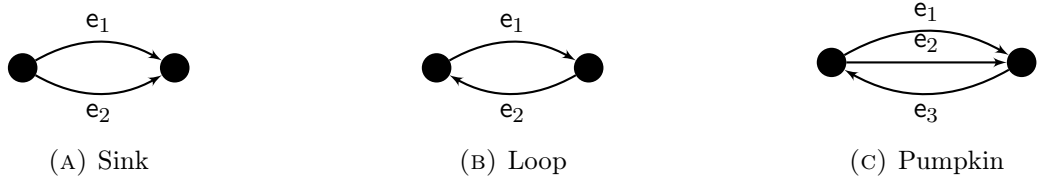

We now only consider boundary conditions that are reflecting the structure of the graph, that is  boundary matrices $V_0, V_1$ from \eqref{eq:loop-bc} should have non-zero entries only if the corresponding endpoints agree. Up to the simultaneous permutation of the rows, 
 we obtain in case of the sink-graph~(A) the following possibilities
 \begin{equation*}
 V_0 = \begin{pmatrix} a & b\\ 0& 0\end{pmatrix},\quad V_1 = \begin{pmatrix} 0 & 0\\ \gamma& \delta\end{pmatrix}
\qquad\text{or}\qquad  V_0 = \begin{pmatrix} a & 0\\ 0& d\end{pmatrix},\quad V_1 = \begin{pmatrix} 0 &0 \\ 0& 0\end{pmatrix}.
\end{equation*} 
Since we assumed negative velocities, \autoref{cor:gen-matrix-det} yields no generation in the first case 
and the generation of a $C_0$-semigroup (but not a group) in the second case if and only if $ad\ne 0$. Note, however, that the second boundary conditions  are degenerated since the second common node (at the endpoints 1) is not taken into account and are usually not considered in the study of transport processes on metric graphs.

In the case of the loop-graph~(B), we have
 \begin{equation*}
 V_0 = \begin{pmatrix} a & 0\\ 0& d\end{pmatrix},\quad V_1 = \begin{pmatrix} 0 &\beta \\ \gamma& 0\end{pmatrix},
\end{equation*} 
which yields the generation of a $C_0$-semigroup  if and only if $ad\ne 0$ and even a $C_0$-group if and only if $ad\beta\gamma\ne 0$. 
\smallskip

Now, let us add another edge, obtaining the so-called pumpkin graph~(C). As before, we assume that   $\lambda_1(\p),\lambda_2(\p),\lambda_3(\p)$ are all negative, and the intervals are parametrized such that the direction of the flow is  from   endpoint $0$ to  endpoint $1$.  
We first take the following coupling conditions for the intervals
\begin{equation*}\label{bc_pumpkin1}
u_1(0) +  u_2(0) =u_3(1)\quad\text{and}\quad  u_3(0)=  u_1(1) +  u_2(1)
\end{equation*}
that incorporate the conservation of the mass in both vertices (i.e., so-called Kirchhoff-type conditions).
These conditions give rise to the boundary matrices 
\[V_0 = \begin{pmatrix} 1 & 1 &0\\ 0& 0&1\\ 0 & 0 & 0\end{pmatrix}\quad\text{and}\quad 
V_1 = \begin{pmatrix} 0 & 0&1 \\1& 1& 0\\ 0 & 0 & 0\end{pmatrix}\]
which yield no generation since $V_0$ is not invertible. This corresponds with the intuition that only 2 boundary conditions for this problem do not suffice. If instead we impose the conditions 
\begin{equation*}\label{bc_pumpkin2}
u_1(0)  = \alpha\cdot u_3(1),\quad   u_2(0) = (1-\alpha)\cdot u_3(1)\quad\text{and}\quad  u_3(0)=  u_1(1) +  u_2(1),
\end{equation*}
for some $\alpha\in (0,1)$, we still have the conservation of mass in the vertices but  the
 boundary matrices now take the form
\[V_0 = \begin{pmatrix} 1 & 0&0\\ 0& 1&0\\ 0 & 0 &1\end{pmatrix}\quad\text{and}\quad 
V_1 = \begin{pmatrix} 0 & 0&\alpha \\0 & 0&1-\alpha \\1& 1& 0\end{pmatrix}\]
and we obtain a well-posed problem governed by a semigroup which is not a group.
\end{exa}

\begin{exa}\label{ex:non-compact}
To obtain simple examples of non-compact graphs we add to the loop-graph from \autoref{ex:loop-graph} two edges of infinite length, parametrized as $\RRp$. For the velocities we consider 3 possibilities,  each yielding a different dimension $\n$ of the respective boundary space $\dX$. They are presented  in \autoref{fig:lasso} together with the corresponding boundary matrices.

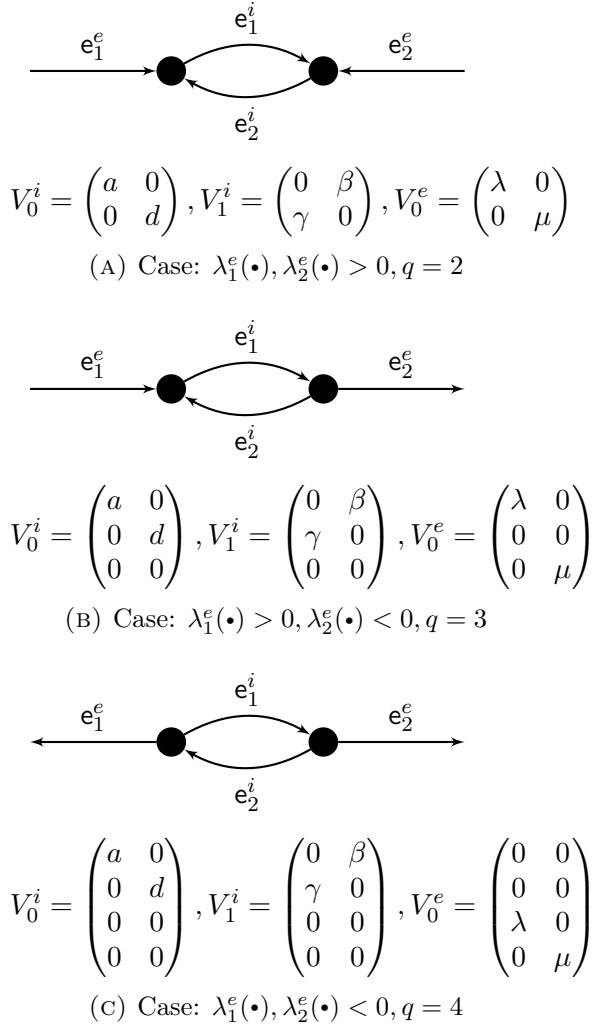
\begin{figure}[!h]  
  \begin{subfigure}[b]{0.45\textwidth}
 \begin{minipage}{.3\linewidth}
    \centering
   \begin{tikzpicture}[-,auto,node distance=3cm,
  thick,main node/.style={circle,fill=black,draw}]

    \node[main node] (1) at (0,0) {};
  \node[main node] (2) at (2,0) {};
  \node (3) at (4,0) {};
  \node (4) at (-2,0){};
  
\tikzset{edge/.style = {->,> = latex'}}
 \draw[edge] (1) edge  [bend left]  node [above] {$\mei_1$}  (2);
 \draw[edge] (2) edge  [bend left]  node {$\mei_2$}  (1); 
 \draw[edge] (3) edge  node [above] {$\mee_2$}  (2); 
  \draw[edge] (4) edge  node [above] {$\mee_1$}  (1); 
\end{tikzpicture}
 \end{minipage}
 \begin{minipage}{.7\linewidth}
  \centering\normalsize
 \[ \Vni = \begin{pmatrix} a &0\\ 0 & d\end{pmatrix}, \Vei =  \begin{pmatrix} 0 &\beta \\ \gamma & 0\end{pmatrix},
  \Vne = \begin{pmatrix} \lambda &0\\0&\mu\end{pmatrix}\]
    \end{minipage}%
\caption{Case: $\lambda^e_1 (\p),\lambda^e_2 (\p)>0, \n=2$}

    \end{subfigure}
      \bigbreak
  \begin{subfigure}[b]{0.45\textwidth}
   \begin{minipage}{.3\linewidth}
     \centering
      \begin{tikzpicture}[-,auto,node distance=3cm,
  thick,main node/.style={circle,fill=black,draw}]

    \node[main node] (1) at (0,0) {};
  \node[main node] (2) at (2,0) {};
  \node (3) at (4,0) {};
  \node (4) at (-2,0){};  
\tikzset{edge/.style = {->,> = latex'}}
 \draw[edge] (1) edge  [bend left]  node [above] {$\mei_1$}  (2);
 \draw[edge] (2) edge  [bend left]  node {$\mei_2$}  (1); 
 \draw[edge] (2) edge  node [above] {$\mee_2$}  (3); 
  \draw[edge] (4) edge  node [above] {$\mee_1$}  (1); 
\end{tikzpicture}
 \end{minipage}
 \begin{minipage}{.7\linewidth}
  \centering\normalsize
 \[ \Vni = \begin{pmatrix} a &0\\ 0 & d\\0&0\end{pmatrix}, \Vei =  \begin{pmatrix} 0 &\beta \\ \gamma & 0\\0&0\end{pmatrix},
  \Vne = \begin{pmatrix} \lambda &0\\0&0\\0&\mu\end{pmatrix}\]
    \end{minipage}%

\caption{Case: $\lambda^e_1 (\p)>0,\lambda^e_2 (\p)<0, \n=3$}
    \end{subfigure}
   \bigbreak   
    \begin{subfigure}[b]{0.45\textwidth}
       \begin{minipage}{.3\linewidth}
     \centering
      \begin{tikzpicture}[-,auto,node distance=3cm,
  thick,main node/.style={circle,fill=black,draw}]

    \node[main node] (1) at (0,0) {};
  \node[main node] (2) at (2,0) {};
  \node (3) at (4,0) {};
  \node (4) at (-2,0){};
  
\tikzset{edge/.style = {->,> = latex'}}
 \draw[edge] (1) edge  [bend left]  node [above] {$\mei_1$}  (2);
 \draw[edge] (2) edge  [bend left]  node {$\mei_2$}  (1); 
 \draw[edge] (2) edge  node [above] {$\mee_2$}  (3); 
  \draw[edge] (1) edge  node [above] {$\mee_1$}  (4); 
\end{tikzpicture}
 \end{minipage}
 \begin{minipage}{.7\linewidth}
  \centering\normalsize
 \[ \Vni = \begin{pmatrix} a &0\\ 0 & d\\0&0\\0&0\end{pmatrix}, \Vei =  \begin{pmatrix} 0 &\beta \\ \gamma & 0\\0&0\\0&0\end{pmatrix},
  \Vne = \begin{pmatrix}0&0\\0&0 \\\lambda &0\\0&\mu\end{pmatrix}\]
    \end{minipage}%
\caption{Case: $\lambda^e_1 (\p),\lambda^e_2 (\p)<0, \n=4$}
    \end{subfigure}   
 \caption{Non-compact graphs from \autoref{ex:non-compact}. The edges are oriented according to the  direction of the flow.}  
  \label{fig:lasso}
\end{figure}
By \autoref{cor:gen-matrix}, we obtain a $C_0$-semigroup in every case if and only if $ad \ne 0$, $ad\mu \ne 0$, and $ad\lambda\nu \ne 0$, respectively.
\end{exa}

\subsection{Compact Graphs}

Starting with \cite{KS:05}, there is a series of articles dealing with transport processes of type \eqref{eq:net-dif} on compact graphs, cf.~\cite{MS07, BN:14, BFN:16}.
 In all these works 
the semigroup approach is used, which is summarized in \cite{DFKNR:10} and \cite[Ch.~18]{BKFR:17}. Note, that in most of these papers the velocities are taken positive and the flow process is directed on the edges from the endpoint 1 to  the endpoint 0.  Moreover, the ``boundary space'' $\dX$ in these works  relates to the vertices of the graph and equals $\CC^n$ while  we take  $\dX=\CC^m$ that complies with the number of the boundary conditions needed (see the discussion preceding \eqref{eq:def-dX}).

The most natural boundary conditions in this context are \emph{Kirchhoff conditions} expressing the conservation of mass, that is, the sum of the incoming flow  in every vertex equals the sum of the outgoing flow from that vertex. 
These conditions are complemented by certain transmission rules that prescribe the proportions of the mass flowing into each of the outgoing edges.  As observed in \cite[Thm.~2.1]{BN:14}, one needs to restrict to the case without sinks in the graph; that is, we have to assume that from every vertex the material is flowing out on at least one edge. In this case, the boundary conditions appearing in the description of the domain of the generator can be rewritten in the form 
\begin{equation}\label{eq:adj}
f(0) = \mathbb{B} f(1),
\end{equation}
where $\mathbb{B}$ is the so-called (weighted, transposed)  \emph{adjacency matrix of the line graph}, i.e.~an $m\times m $ matrix incorporating the structure of the graph, see \cite[Prop.~18.2]{BKFR:17} for more details.  By \autoref{cor:gen-matrix-det}, the network transport problem \eqref{eq:net-dif} with boundary conditions of the form \eqref{eq:adj} is well-posed with no extra assumptions on the matrix $\BB$. 

Our results cover all situations where the boundary conditions can be  written as in \eqref{eq:adj} which,  apart from simple transport equation with constant coefficients like in \cite{KS:05,BN:14,BFN:16,BKFR:17}, include also non-constant coefficients cases as in \cite{MS07,Ban:16} or transport equations with scattering, see \cite{Rad08,DFKNR:10}. Yet another example of this kind are transmission line networks with non-conservative junction conditions as studied in \cite{Car:11}.  Since every edge has exactly one starting and one ending vertex, the local boundary conditions 
 \cite[(3.3)]{Car:11} can be equivalently written as \eqref{eq:adj} where the entries of $\BB$ correspond to appropriate entries of locally given scattering matrices. Again, no additional assumption on $\BB$ is needed in order to obtain well-posedness. In \cite{Car:11} only balanced graphs are considered (i.e., graphs with the same number of incoming and outgoing edges in each vertex) and  invertibility of the scattering matrices is assumed, however, the focus there is not on well-posedness but on spectral properties of the generator. 

We point out that  \autoref{thm:gen-Lp} and its corollaries for the compact case, i.e.,
\autoref{cor:gen-group} and \autoref{cor:gen-matrix-det}, enable to tackle problems with much more general boundary conditions as studied in the references mentioned above. Moreover, one can apply the spectral theory developed in \cite{AE:18} to characterize the spectrum of the generator in terms of the graph matrix $\BB$, see, e.g.~\autoref{lem:sigmaG}  and \cite[Cor.~18.13]{BKFR:17}.

\subsection{Non-compact Graphs}

To our knowledge, first-order equations on non-compact graphs have not attained much attention in the literature.
In \cite{Exn:13} the action of momentum operator $-\imath \hbar \tdds$  on non-compact metric (or quantum) graphs has been considered in the  $\rL^2$-setting,  the conditions yielding a self-adjoint generator were presented, and its spectral properties studied. Further, \cite{JW:19}  is devoted to port-Hamiltonian systems corresponding to a class of hyperbolic partial differential equations of first order  on the semi-axis  and the main result is the characterization of the boundary conditions yielding $C_0$-semigroup generation. On the other hand, there has been some more work done in the case of  second-order equations, see  \cite{EKF:19} and the references therein.

Let us first show how to write the boundary conditions also in the non-compact case in the form resembling \eqref{eq:adj}. 
Starting with the problem on a compact graph as in \autoref{cor:gen-matrix-det} and assuming that all the velocities are negative, the generation condition becomes the invertibility of  $V_0$. For invertible $V_0$, we  can write the  boundary conditions as
\[V_0 f(0)= V_1 f(1) \iff f(0) = V_0^{-1} V_1 f(1),\]
hence, by taking $\BB := V_0^{-1} V_1$, we obtain \eqref{eq:adj}. We now  proceed similarly in the non-compact case, with the setting as in \autoref{cor:gen-matrix} and taking all the internal velocities to be negative. We can write the boundary conditions in $D(G)$ as 
\[\bigl(\Vne,\Vni\bigr) { \fe(0)\choose \fin(0) } = \Vei\fin(1) \iff 
 { \fe(0)\choose \fin(0) } = \bigl(\Vne,\Vni\bigr) ^{+}  \Vei\fin(1)
\]
where $A^+$ denotes the Moore-Penrose  pseudoinverse of matrix $A$. Note, that the generation condition in this case yields that the matrix 
$ \bigl(\Vne,\Vni\bigr) $ has full rank ($=\n$) and the pseudoinverse is the right inverse of this matrix. Taking $\BB := \bigl(\Vne,\Vni\bigr) ^{+}  \Vei$ we obtain the desired form of the conditions. Thus, by \autoref{cor:gen-matrix} we can cover a wide variety of boundary conditions. 
\begin{exa}\label{ex:non-compact-2}
Take a non-compact graph with 5 finite and 4 infinite edges
as  depicted on \autoref{fig:non-compact-2}. 
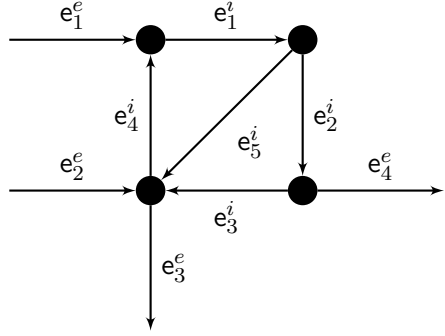
\begin{figure}[!h]  
     \centering
      \begin{tikzpicture}[-,auto,node distance=3cm,
  thick,main node/.style={circle,fill=black,draw}]

\node[main node] (1) at (0,0) {};
\node[main node] (2) at (2,0) {};
\node[main node] (3) at (2,2) {};
\node[main node] (4) at (0,2) {};
\node(5) at (-2,2) {};
\node(6) at (-2,0) {};
\node (7) at (0,-2) {};
\node(8) at (4,0) {};

\tikzset{edge/.style = {->,> = latex'}}
 \draw[edge] (4) edge  node  {$\mei_1$}  (3);
 \draw[edge] (3) edge  node {$\mei_2$}  (2); 
 \draw[edge] (2) edge  node   {$\mei_3$}  (1); 
 \draw[edge] (1) edge  node {$\mei_4$}  (4); 
 \draw[edge] (3) edge  node {$\mei_5$}  (1); 
 \draw[edge] (5) edge  node  {$\mee_1$}  (4); 
  \draw[edge] (6) edge  node  {$\mee_2$}  (1); 
 \draw[edge] (1) edge  node  {$\mee_3$}  (7); 
 \draw[edge] (2) edge  node  {$\mee_4$}  (8); 
\end{tikzpicture}
\caption{Non-compact graph from \autoref{ex:non-compact-2}. The edges are oriented according to the  direction of the flow.}  
  \label{fig:non-compact-2}
 \end{figure}
We assume that all internal velocities $\lambda^i_k(\p) <0$, $k=1,\ldots,5$. For the external part we assume
\[\lambda^e_1(\p)=\lambda^e_2(\p) >0, \quad \lambda^e_3(\p)=\lambda^e_4(\p) <0,\]
hence $\Pip=0$, $\Pim = Id$ and $\rank \Pem = 2$, so $\n=7$.
We equip the system with the following coupling conditions in the common vertices
\begin{alignat*}{3}
u_2^i(0) &= \alpha_1 u_1^i(1),  &\qquad   u_1^i(0)   &= u_4^i(1) + u_1^e(0),\\
u_5^i(0) &= \alpha_2 u_1^i(1),  &\qquad  u_4^i(0) &= \gamma_1 u_5^i(1) +\delta_1 u_3^i(1) + \varepsilon_1 u_2^e(0),\\
u_3^i(0) &= \beta_1 u_2^i(1), &\qquad   u_3^e(0) &= \gamma_2 u_5^i(1) +\delta_2 u_3^i(1) + \varepsilon_2 u_2^e(0),\\
u_4^e(0) &= \beta_2 u_2^i(1)  &&  
\end{alignat*}
for some coefficients $\alpha_1,\alpha_2,\beta_1,\beta_2,\gamma_1,\gamma_2, \delta_1,\delta_2, \varepsilon_1, \varepsilon_2 >0$. The corresponding boundary matrices in this situation are
\[ 
\Vni=  \begin{pmatrix} 1 & 0&0&0&0\\  0&1&0&0&0\\  0&0&1&0&0\\  0&0&0&1&0\\  0&0&0&0&1\\  0&0&0&0&0\\  0& 0&0&0&0\end{pmatrix},
\qquad 
\Vei=  \begin{pmatrix} 0&0&0&1&0\\  \alpha_1 &0&0&0&0\\  0&\beta_1&0&0&0\\  0&0&\gamma_1&0&\delta_1\\  \alpha_2&0&0&0&0\\  0&0&\gamma_2&0&\delta_2\\  0& \beta_2&0&0&0\end{pmatrix},\qquad
\Vne=  \begin{pmatrix} 1 & 0&0&0\\  0&0&0&0\\  0&0&0&0\\  0&\varepsilon_1 &0&0\\  0&0&0&0\\  0&\varepsilon_2&1&0&\\  0&0&0&1\end{pmatrix}.
\]
Observe that the matrix 
\[\sR_0=\bigl(\Vne,\Vei\Pip-\Vni\Pim\bigr) = \bigl(\Vne,\Vni\bigr) \]
from \eqref{eq:R0-matrix} is of full rank, hence we obtain the generation of a $C_0$-semigroup.
\end{exa}

\begin{exa}\label{ex:star} The non-compact graphs considered in the literature are usually assumed to be star-shaped, that is without the compact part. 
Let us thus consider transport equations on a star-shaped network  with  a single vertex and $\ell$ infinite edges. We assume the material is transported into the common node along $\ell-q$ edges while it flows out of it along  $q$ edges.   We orient and enumerate the edges according to the  flow direction:
\[\lambda^e_i <0, \;   i = 1,\dots, q \quad \text{and} \quad \lambda^e_j>0, \; j =  q+1, \dots, \ell, \]
as  depicted in \autoref{fig:star}. 
\begin{figure}[!h]  
     \centering
      \begin{tikzpicture}[-,auto,node distance=3cm,
  thick,main node/.style={circle,fill=black,draw}]

  \node[main node] (1) at (0,0) {};
  \node(2) at (-3,2.1) {};
  \node(3) at (-3.5,1) {};
  \node (4) at (-3,-1.7) {};
  \node(5) at (3,2.1) {};
  \node(6) at (3,-1.7) {};
  \node(7) at (3.5,1) {};

\tikzset{edge/.style = {->,> = latex'}}
 \draw[edge] (2) edge  node  {$\mee_{q+1}$}  (1);
 \draw[edge] (3) edge  node  [anchor=south] {$\mee_{q+2}$}  (1); 
 \draw[edge] (4) edge  node  [anchor=north]   {$\mee_{\ell}$}  (1); 
 \draw[edge] (1) edge  node {$\mee_{q}$}  (5); 
  \draw[edge] (1) edge  node [anchor=north] {$\mee_{1}$}  (6); 
   \draw[edge] (1) edge  node [anchor=south]{$\mee_{q-1}$}  (7);  
  \foreach \y in {20,15,10,5,0,-5,-10}{
\draw[fill] (\y:-2) circle (.5pt);};
  \foreach \x in {10,5,0,-5,-10,-15,-20}{
\draw[fill] (\x:2) circle (.5pt);} 
\end{tikzpicture}
 \caption{Star-shaped non-compact graph from \autoref{ex:star}. The edges are oriented according to the  direction of the flow.}  
  \label{fig:star}
 \end{figure}
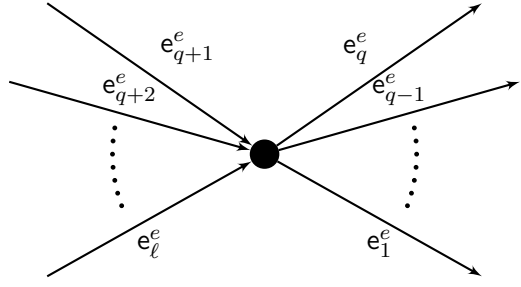
 In the common vertex we impose the boundary conditions for the outgoing flows in the form
\[ u_i^e(t,0) = \sum_{j=q+1}^{\ell} \alpha_{ij} u_j^e(t,0),\quad i =  1, \dots, q. \]
Since for any  choice of the coefficients  $\alpha_{ij}\in\RR$ we have $\rank(\Pem) = q = \rank\Vne$,  \autoref{cor:gen-matrix} yields the generation of a $C_0$-semigroup.
This result generalizes \cite[Ex.~6.2]{JW:19}.
\end{exa}
\goodbreak
\appendix
\label{app:dp}

\section{Domain Perturbation for Generators of $C_0$-Semigroups}
\label{sec:dom-pert}

In this appendix we first briefly recall a perturbation result from \cite[App.~A]{EKF:19} which easily follows 
from \cite[Sect.~4.3]{ABE:13} and is our main tool to prove \autoref{thm:gen-Lp}. 
Further, in \autoref{thm:NC-gen} we supplement this theory with a new result which is needed to obtain the necessary condition in \autoref{thm:NC-1}.
\smallskip

To explain the general setup we consider
\begin{itemize}
\item two Banach spaces $X$ and $\partial X$, called ``state'' and ``boundary'' spaces, respectively;
\item a closed, densely defined ``maximal'' operator\footnote{``maximal'' concerns the size of the domain, e.g., a differential operator without boundary conditions.} $A_m:D(A_m)\subseteq X\to X$;
\item the Banach space $[D(A_m)]:=(D(A_m),\|\p\|_{A_m})$ where $\|f\|_{A_m}:=\|f\|+\|A_mf\|$ is the graph norm;
\item two ``boundary'' operators $L,\Phi\in\sL([D(A_m)],\partial X)$.
\end{itemize}
Then,  define two restrictions $A,\,G\subset A_m$ by
\begin{align*}
D(A):&=\bigl\{f\in D(A_m):Lf=0\bigr\}=\ker(L),\\
D(G):&=\bigl\{f\in D(A_m): \Phi f=0 \bigr\}=\ker(\Phi).
\end{align*}

Hence, one can consider $G$ with boundary condition $\Phi f=0$ as a perturbation of the operator $A$ with abstract ``Dirichlet type'' boundary condition $Lf=0$.
In order to proceed we make throughout this appendix the following 

\begin{asu}\label{asu:BP}
\makeatletter
\hyper@anchor{\@currentHref}%
\makeatother
\begin{enumerate}[(i)]
\item The operator $A$ generates a $C_0$-semigroup $\Tt$ on $X$;
\item the boundary operators $L,\Phi:D(A_m)\to\partial X$ are surjective.
\end{enumerate}
\end{asu}

By combining \cite[Cor.~A.5]{EKF:19} and \cite[Lem.~A.6]{EKF:19} we obtain the following.

\begin{thm}\label{thm:pert-bc-v3}
Assume that there exist $1\le p<+\infty$, $\tn>0$, $M\ge0$,  and a strongly continuous family $\Bttn\subset\sL(\rL^p([0,\tn],\dX),X)$  such that the following assertions hold.
\begin{enumerate}[(i)]
\item For every $v\in\rW^{2,p}_0([0,\tn],\dX)$ the function
\begin{equation*}
x\colon [0,\tn]\to X,\quad x(t):=\sBt v
\end{equation*}
is a classical solution of the boundary control problem
\begin{equation*}
\begin{cases}
\dot x(t)=A_m x(t),&0\le t\le\tn,\\
L x(t)=v(t),&0\le t\le\tn,\\
x(0)=0.
\end{cases}
\end{equation*}
\item  For all $x\in D(A),$
\[\int_0^{t_0}\bigl\|\Phi\,T(s)x\bigr\|_\dX^p\ds \leq M\cdot\|x\|_X^p.\]
\item For all $v\in\rW_0^{2,p}([0,t_0],\dX)$,
\[\int_0^{\tn}\bigl\|\Phi \sBt v\bigr\|_\dX^p\dt\le M\cdot\|v\|_p^p.\]
\item $\sQ_{\tn}$  is invertible, where $\sQ_{\tn}\in\sL\bigl(\rL^p([0,t_0],\dX)\bigr)$ is given by 
\begin{equation*} 
(\sQ_{\tn} v)(\p)=\Phi {\mathcal B}_{\p} v
\quad\text{for all }v\in\rW_0^{2,p}([0,t_0],\dX).
\end{equation*}
\end{enumerate}
Then $G\subset A_m$, $D(G):=\ker(\Phi)$ generates a $C_0$-semigroup on the Banach space $X$.
\end{thm}

\begin{rem}\label{rem:C-admiss}
If assumption~(ii) in \autoref{thm:pert-bc-v3} is satisfied, then the operator $\Phi$ is called \emph{$p$-admissible} for $\Tt$. 
\end{rem}

For more details and examples regarding boundary perturbation results of the above type we refer to \cite{ABE:13, ABE:15,BE:13}.

While the previous result provides a sufficient condition for $G\subset A_m$ with domain $D(G)=\ker\Phi$ to be the generator of a $C_0$-semigroup our next aim is to give a necessary one. To this end we first need the following result from in \cite[Lem.~1.2]{Gre:87}.

\begin{lem}\label{lem:Ll}
 For each $\lambda\in\rho(A)$ the restricted operator $L|_{\ker(\lambda-A_m)}$ is invertible and
$L_\lambda:=(L|_{\ker(\lambda-A_m)})^{-1}:\partial X\to\ker(\lambda-A_m)\subseteq X$
is bounded.
\end{lem}

Using these so-called abstract Dirichlet operators $L_\lambda\in\sL(\dX,X)$ one can characterize the spectrum of $G$, cf.\  \cite[Cor.~3.13]{AE:18}. Note that $L_\lambda\in\sL(\dX,[D(A_m)])$, hence $\Phi L_\lambda:\dX\to\dX$ is bounded for every $\lambda\in\rho(A)$.

\begin{lem}\label{lem:sigmaG}
For every $\lambda\in\rho(A)$ one has
\[
\lambda\in\rho(G)\quad\iff\quad \Phi L_\lambda\in\sL(\dX)\text{ is invertible}
\]
\end{lem}

Since for  $\lambda\in\rho(A)$ the norms $\|\cdot\|_{[D(A_m)]}$ and $\|\cdot\|_X$ restricted to $\rg(L_\lambda)=\ker(\lambda-A_m)$ are equivalent, the restriction $L:\bigl(\ker(\lambda-A_m),\|\cdot\|_X\bigr)\to\dX$ is bounded. Hence, 
\begin{equation}\label{eq:def-nLl}
\|L\|_\lambda:=\|L\|_{\sL((\ker(\lambda-A_m),\|\cdot\|_X),\dX)}<+\infty
\end{equation}
is well-defined.
Now we are well-prepared to give a necessary condition for $G$ to be a generator on $X$.

\begin{thm}\label{thm:NC-gen}
Suppose that $G\subset A_m$ with domain $D(G)=\ker(\Phi)$ is a generator of a $C_0$-semigroup where
$\Phi=\Phi_0+V\cdot L$ for $\Phi_0\in\sL([D(A_m)],\partial X)$ and a non-invertible operator $V\in\sL(\dX)$. If  $\dim(\dX)<\infty$ then
\begin{equation}\label{eq:NC-Gen}
\lim_{\lambda\to+\infty}\|L\|_\lambda\cdot\|\Phi_0L_\lambda\|=+\infty.
\end{equation}
\end{thm}

\begin{proof} Let $\lambda,\mu\in\rho(A)\cap\rho(G)$. Then, by \autoref{lem:Ll}, the Dirichlet operators $L_\lambda$, $L_\mu$ exist and by \autoref{lem:sigmaG} the operator $\Phi L_\lambda\in\sL(\dX)$ is invertible. A simple computation then shows that 
\[
R(\lambda,G)L_\mu =\tfrac1{\lambda-\mu}\bigl(L_\mu-L_\lambda(\Phi L_\lambda)^{-1}\Phi L_\mu \bigr).
\]
Using that $\lambda R(\lambda,G)f\to f$ as $\lambda\to+\infty$ for every $f\in X$ we conclude that for all $x\in\dX$
\[
\lim_{\lambda\to+\infty}L_\lambda(\Phi L_\lambda)^{-1}\Phi L_\mu x=0.
\]

Since $\Phi L_\mu\in\sL(\dX)$ is invertible and the unit ball in $\dX$ is compact this implies
\begin{equation}\label{eq:limLlPLli}
\lim_{\lambda\to+\infty}\bigl\|L_\lambda(\Phi L_\lambda)^{-1}\bigr\|=0.
\end{equation}
Next, using that $LL_\lambda=\Id_{\dX}$ from the decomposition $\Phi=\Phi_0+V\cdot L$ we obtain
\[
V=\Phi L_\lambda-\Phi_0L_\lambda=\Phi L_\lambda\cdot\bigl(\Id-(\Phi L_\lambda)^{-1}\cdot\Phi_0L_\lambda\bigr)\in\sL(\dX)
\]
which by assumption is not invertible. Therefore,
\begin{equation}
\|\Phi_0L_\lambda\|\ge\bigl\|(\Phi L_\lambda)^{-1}\bigr\|^{-1}.
\end{equation}
Combining this with \eqref{eq:def-nLl} gives
\begin{align*}
\|\Phi_0L_\lambda\|^{-1}&\le\bigl\|(\Phi L_\lambda)^{-1}\bigr\|\\
&=\bigl\|LL_\lambda(\Phi L_\lambda)^{-1}\bigr\|_{\sL(X)}\\
&\le\|L\|_\lambda\cdot\bigl\| L_\lambda(\Phi L_\lambda)^{-1}\bigr\|_{\sL(\dX,X)}.
\end{align*}
By \eqref{eq:limLlPLli} we infer
\[
\lim_{\lambda\to+\infty}\|L\|_\lambda^{-1}\cdot\|\Phi_0L_\lambda\|^{-1}=0
\]
which implies \eqref{eq:NC-Gen}.
\end{proof}

Note that in the previous proof we didn't  make use of the fact that $A$ is a generator but only needed the existence of the Dirichlet operators $L_\lambda$ for $\lambda$ sufficiently large.


\end{document}